\documentclass[11pt,leqno,a4paper]{amsart}
\usepackage{amssymb,amscd,amsmath}
\usepackage{pspicture}
\usepackage{graphicx}

\usepackage{mathptmx}

\usepackage[matrix,arrow,curve,frame]{xy}    

\xymatrixcolsep{1.9pc}                          
\xymatrixrowsep{1.9pc}
\newdir{ >}{{}*!/-5pt/\dir{>}}                  

 \calclayout

\makeatletter

\renewcommand{\subsection}{\@startsection{subsection}{1}{0pt}{-3.25ex plus -1ex minus-.2ex}{1.5ex plus.2ex}{\normalfont\it}}
\renewcommand{\section}{\@startsection{section}{1}{\parindent}{3.5ex plus 1ex minus .2ex}{2.3ex plus.2ex}{\sc}}

\renewcommand{\phi}{\varphi}
\renewcommand{\leq}{\leqslant}
\renewcommand{\geq}{\geqslant}
\renewcommand{\epsilon}{\varepsilon}

\renewcommand{\kappa}{\varkappa}

\DeclareMathOperator{\spec}{Spec}
 \DeclareMathOperator{\cyl}{cyl}

 \DeclareMathOperator{\mot}{mot}

 \DeclareMathOperator{\cone}{cone}
\DeclareMathOperator{\Hom}{Hom} 
 
 \DeclareMathOperator{\id}{id}

 \DeclareMathOperator{\Mor}{Mor}
 \DeclareMathOperator{\colim}{colim}

 \DeclareMathOperator{\kr}{Ker}
 \DeclareMathOperator{\im}{Im}
 \DeclareMathOperator{\nis}{nis}

 \DeclareMathOperator{\Mod}{Mod}
 \DeclareMathOperator{\Ob}{Ob}
\DeclareMathOperator{\supp}{Supp}

\newcommand {\lp}{\colim}

\newcommand{\lra}[1]{\bl{#1}\longrightarrow\relax}
\newcommand{\bl}[1]{\buildrel #1\over}
\newcommand{\cc}{\mathcal}
\newcommand{\bb}{\mathbb}

\newcommand{\ff}{\mathfrak}

\newcommand{\op}{{\textrm{\rm op}}}

\newcommand{\wt}{\widetilde}

\newcommand{\sheff}{SH^{\mot}_{S^1}}
\newcommand{\shnis}{SH^{\nis}_{S^1}}

\newtheorem{thm}{Theorem}[section]
\newtheorem{prop}[thm]{Proposition}

\newtheorem{cor}[thm]{Corollary}
\newtheorem{lem}[thm]{Lemma}

\newtheorem{rem}[thm]{Remark}

\newtheorem{defs}[thm]{Definition}
\newtheorem{notn}[thm]{Notation}

\makeatother

\begin{document}

\footskip30pt


\title{The triangulated category of K-motives $DK_{-}^{\textrm{\rm eff}}(k)$}
\author{Grigory Garkusha}
\address{Department of Mathematics, Swansea University, Singleton Park, Swansea SA2 8PP, United Kingdom}
\email{g.garkusha@swansea.ac.uk}

\author{Ivan Panin}
\address{St. Petersburg Branch of V. A. Steklov Mathematical Institute,
Fontanka 27, 191023 St. Petersburg, Russia}

\address{St. Petersburg State University, Department of Mathematics and Mechanics, Universitetsky prospekt, 28, 198504,
Peterhof, St. Petersburg, Russia}
\email{paniniv@gmail.com}

\thanks{Supported by EPSRC grant EP/J013064/1.}



\begin{abstract}
For any perfect field $k$ a triangulated category of $K$-motives
$DK_{-}^{\text{eff}}(k)$ is constructed in the style of Voevodsky's
construction of the category $DM_-^{\text{eff}}(k)$. To each smooth
$k$-variety $X$ the $K$-motive $M_{\bb K}(X)$ is associated in the
category $DK_{-}^{\text{eff}}(k)$ and
   $$K_n(X)=\Hom_{DK_{-}^{\text{eff}}(k)}(M_{\bb K}(X)[n],M_{\bb K}(pt)),\quad n\in\bb Z,$$
where $pt=\spec(k)$ and $K(X)$ is Quillen's $K$-theory of $X$.
\end{abstract}
\maketitle

\thispagestyle{empty} \pagestyle{plain}

\newdir{ >}{{}*!/-6pt/@{>}} 


\section{Introduction}

The Voevodsky triangulated category of motives
$DM_-^{\text{eff}}(k)$~\cite{Voe1} provides a natural framework to
study motivic cohomology. In~\cite{GP} the authors constructed a
triangulated category of $K$-motives providing a natural framework
for Grayson's motivic spectral sequence~\cite{Gr}
   $$E_2^{pq}=H_{\cc M}^{p-q,-q}(X,\bb Z)\Longrightarrow K_{-p-q}(X)$$
that relates the motivic cohomology groups of a smooth variety $X$
to its algebraic $K$-groups. The main idea was to use a kind of
motivic algebra of spectral categories and modules over them.

In this paper an alternative approach to constructing a triangulated
category of $K$-motives is presented. We work in the framework of
strict $V$-spectral categories introduced in the paper
(Definition~\ref{vsp}). The main feature of such a spectral category
$\cc O$ is that it is connective and Nisnevich excisive in the sense
of~\cite{GP}, and $\pi_0\cc O$-(pre)sheaves, where $\pi_0\cc O$ is a
ringoid associated to $\cc O$, share lots of common properties with
(pre)sheaves with transfers (or $Cor$-(pre)sheaves) in the sense of
Voevodsky~\cite{Voe}.

To any strict $V$-spectral category over $k$-smooth varieties we
associate a triangulated category $D\cc O_-^{\text{eff}}(k)$, which
in spirit is constructed similarly to $DM_-^{\text{eff}}(k)$
(Section~\ref{dominus}). For instance, the ringoid of
correspondences $Cor$ gives rise to a strict $V$-spectral category
$\cc O=\cc O_{cor}$ whenever the base field $k$ is perfect. In this
case the Voevodsky category $DM_-^{\text{eff}}(k)$ is recovered as
the category $D\cc O_-^{\text{eff}}(k)$ (Corollary~\ref{kongo}).

The main $V$-spectral category $\bb K$ is constructed in
Section~\ref{CategoriesA} (see Theorem~\ref{semga}). It is strict
over perfect fields. The associated triangulated category $D\cc
O_-^{\text{eff}}(k)$ is denoted by $DK_-^{\text{eff}}(k)$. The
spectral category $\bb K$ is a priori different from spectral
categories constructed by the authors in~\cite{GP}. But we expect
that associated motivic model categories of modules over the
spectral categories are equivalent.

To each smooth $k$-variety $X$ we associate its $K$-motive $M_{\bb
K}(X)$. By definition, it is an object of the category
$DK_{-}^{\text{eff}}(k)$. We prove in Theorem~\ref{vesmaneploho}
that
   $$K_n(X)=\Hom_{DK_{-}^{\text{eff}}(k)}(M_{\bb K}(X)[n],M_{\bb K}(pt)),\quad n\in\bb Z,$$
where $pt=\spec(k)$ and $K(X)$ is Quillen's $K$-theory of $X$. Thus
Quillen's $K$-theory is represented by the $K$-motive of the point.

The spectral category $\bb K$ is of great utility in authors'
paper~\cite{GP1}, in which they solve some problems related to the
motivic spectral sequence. In fact, the problems were the main
motivation for constructing the spectral category $\bb K$ and
developing the machinery of $K$-motives.

Throughout the paper we denote by $\text{Sm}/k$ the category of
smooth separated schemes of finite type over the base field $k$.

\section{Preliminaries}

We work in the framework of spectral categories and modules over
them in the sense of Schwede--Shipley~\cite{SS}. We start with
preparations.

Recall that symmetric spectra have two sorts of homotopy groups
which we shall refer to as {\it naive\/} and {\it true homotopy
groups\/} respectively following terminology of~\cite{Sch}.
Precisely, the $k$th naive homotopy group of a symmetric spectrum
$X$ is defined as the colimit
   $$\hat\pi_k(X)=\colim_n\pi_{k+n}X_n.$$
Denote by $\gamma X$ a stably fibrant model of $X$ in $Sp^\Sigma$.
The $k$-th true homotopy group of $X$ is given by
   $$\pi_kX=\hat\pi_k(\gamma X),$$
the naive homotopy groups of the symmetric spectrum $\gamma X$.

Naive and true homotopy groups of $X$ can be considerably different
in general (see, e.g.,~\cite{HSS,Sch}). The true homotopy groups
detect stable equivalences, and are thus more important than the
naive homotopy groups. There is an important class of {\it
semistable\/} symmetric spectra within which
$\hat\pi_*$-isomorphisms coincide with $\pi_*$-isomorphisms. Recall
that a symmetric spectrum is semistable if some (hence any) stably
fibrant replacement is a $\pi_*$-isomorphism. Suspension spectra,
Eilenberg--Mac Lane spectra, $\Omega$-spectra or $\Omega$-spectra
from some point $X_n$ on are examples of semistable symmetric
spectra (see~\cite{Sch}). So Waldhausen's algebraic $K$-theory
symmetric spectrum, which we shall use later, is semistable.
Semistability is preserved under suspension, loop, wedges and shift.

A symmetric spectrum $X$ is {\it $n$-connected\/} if the true
homotopy groups of $X$ are trivial for $k\leq n$. The spectrum $X$
is {\it connective\/} if it is $(-1)$-connected, i.e., its true
homotopy groups vanish in negative dimensions. $X$ is {\it bounded
below\/} if $\pi_i(X)=0$ for $i\ll 0$.

\begin{defs}\label{spectral}{\rm
(1) Following~\cite{SS} a {\it spectral category\/} is a category
$\cc O$ which is enriched over the category $Sp^\Sigma$ of symmetric
spectra (with respect to smash product, i.e., the monoidal closed
structure of \cite[2.2.10]{HSS}). In other words, for every pair of
objects $o,o'\in\cc O$ there is a morphism symmetric spectrum $\cc
O(o,o')$, for every object $o$ of $\cc O$ there is a map from the
sphere spectrum $S$ to $\cc O(o,o)$ (the ``identity element" of
$o$), and for each triple of objects there is an associative and
unital composition map of symmetric spectra $\cc O(o',o'')\wedge\cc
O(o,o') \to\cc O(o,o'')$. An $\cc O$-module $M$ is a contravariant
spectral functor to the category $Sp^\Sigma$ of symmetric spectra,
i.e., a symmetric spectrum $M(o)$ for each object of $\cc O$
together with coherently associative and unital maps of symmetric
spectra $M(o)\wedge\cc O(o',o)\to M(o')$ for pairs of objects
$o,o'\in\cc O$. A morphism of $\cc O$-modules $M\to N$ consists of
maps of symmetric spectra $M(o)\to N(o)$ strictly compatible with
the action of $\cc O$. The category of $\cc O$-modules will be
denoted by $\Mod\cc O$.

(2) A {\it spectral functor\/} or a {\it spectral homomorphism\/}
$F$ from a spectral category $\cc O$ to a spectral category $\cc O'$
is an assignment from $\Ob\cc O$ to $\Ob\cc O'$ together with
morphisms $\cc O(a,b)\to\cc O'(F(a),F(b))$ in $Sp^\Sigma$ which
preserve composition and identities.

(3) The {\it monoidal product\/} $\cc O\wedge\cc O'$ of two spectral
categories $\cc O$ and $\cc O'$ is the spectral category where
$\Ob(\cc O\wedge\cc O'):=\Ob\cc O\times\Ob\cc O'$ and $\cc
O\wedge\cc O'((a,x),(b,y)):= \cc O(a,b)\wedge\cc O'(x,y)$.


(4) A spectral category $\cc O$ is said to be {\it connective\/} if
for any objects $a,b$ of $\cc O$ the spectrum $\cc O(a,b)$ is
connective.

(5) By a {\it ringoid\/} over $\text{Sm}/k$ we mean a preadditive
category $\cc R$ (i.e., a category enriched over abelian groups)
whose objects are those of $\text{Sm}/k$ together with a functor
   $$\rho:\text{Sm}/k\to\cc R,$$
which is identity on objects. Every such ringoid gives rise to a
spectral category $\cc O_{\cc R}$ whose objects are those of
$\text{Sm}/k$ and the morphisms spectrum $\cc O_{\cc R}(X,Y)$,
$X,Y\in \text{Sm}/k$, is the Eilenberg--Mac~Lane spectrum $H\cc
R(X,Y)$ associated with the abelian group $\cc R(X,Y)$. Given a map
of schemes $\alpha$, its image $\rho(\alpha)$ will also be denoted
by $\alpha$, dropping $\rho$ from notation.

(6) Let $\cc O_{naive}$ be the spectral category whose objects are
those of $\text{Sm}/k$ and morphism spectra are defined as
   $$\cc O_{naive}(X,Y)_p=\Hom_{\text{Sm}/k}(X,Y)_+\wedge S^p$$
for all $p\geq 0$ and $X,Y\in \text{Sm}/k$. By a {\it spectral
category over\/} $\text{Sm}/k$ we mean a pair $(\cc O,\sigma)$,
where $\cc O$ is a spectral category whose objects are those of
$\text{Sm}/k$ and
   $$\sigma:\cc O_{naive}\to\cc O$$
is a spectral functor which is identity on objects. If there is no
likelihood of confusion, we shall drop $\sigma$ from notation.

}\end{defs}

\begin{rem}{\rm
It is straightforward to verify that the category of $\cc
O_{naive}$-modules can be regarded as the category of presheaves
$Pre^\Sigma(\text{Sm}/k)$ of symmetric spectra on $\text{Sm}/k$.
This is used in the sequel without further comment.

}\end{rem}

Let $\cc O$ be a spectral category and let $\Mod\cc O$ be the
category of $\cc O$-modules. Recall that the projective stable model
structure on $\Mod\cc O$ is defined as follows (see~\cite{SS}). The
weak equivalences are the objectwise stable weak equivalences and
fibrations are the objectwise stable projective fibrations. The
stable projective cofibrations are defined by the left lifting
property with respect to all stable projective acyclic fibrations.

Let $\cc Q$ denote the set of elementary distinguished squares in
$\text{Sm}/k$ (see~\cite[3.1.3]{MV})
   \begin{equation*}\label{squareQ}
    \xymatrix{\ar@{}[dr] |{\textrm{$Q$}}U'\ar[r]\ar[d]&X'\ar[d]^\phi\\
              U\ar[r]_\psi&X}
   \end{equation*}
and let $\cc O$ be a spectral category over $\text{Sm}/k$ in the
sense of Definition~\ref{spectral}(6). By $\cc Q_{\cc O}$ denote the
set of squares
   \begin{equation*}\label{squareOQ}
    \xymatrix{\ar@{}[dr] |{\textrm{$\cc O Q$}}\cc O(-,U')\ar[r]\ar[d]&\cc O(-,X')\ar[d]^\phi\\
              \cc O(-,U)\ar[r]_\psi&\cc O(-,X)}
   \end{equation*}
which are obtained from the squares in $\cc Q$ by taking $X\in
\text{Sm}/k$ to $\cc O(-,X)$. The arrow $\cc O(-,U')\to\cc O(-,X')$
can be factored as a cofibration $\cc O(-,U')\rightarrowtail Cyl$
followed by a simplicial homotopy equivalence $Cyl\to\cc O(-,X')$.
There is a canonical morphism $A_{\cc O Q}:=\cc O(-,U)\bigsqcup_{\cc
O(-,U')} Cyl\to\cc O(-,X)$.

\begin{defs}[see~\cite{GP}]{\rm
I. The {\it Nisnevich local model structure\/} on $\Mod\cc O$ is the
Bousfield localization of the stable projective model structure with
respect to the family of projective cofibrations
   \begin{equation*}\label{no}
    \cc N_{\cc O}=\{\cyl(A_{\cc O Q}\to\cc O(-,X))\}_{\cc Q_{\cc O}}.
   \end{equation*}
The homotopy category for the Nisnevich local model structure will
be denoted by $\shnis\cc O$. In particular, if $\cc O=\cc O_{naive}$
then we have the Nisnevich local model structure on
$Pre^\Sigma(\text{Sm}/k)=\Mod\cc O_{naive}$ and we shall write
$\shnis(k)$ to denote $\shnis\cc O_{naive}$.

II. The {\it motivic model structure\/} on $\Mod\cc O$ is the
Bousfield localization of the Nisnevich local model structure with
respect to the family of projective cofibrations
   \begin{equation*}\label{ao}
    \cc A_{\cc O}=\{\cyl(\cc O(-,X\times\bb A^1)\to\cc O(-,X))\}_{X\in \text{Sm}/k}.
   \end{equation*}
The homotopy category for the motivic model structure will be
denoted by $\sheff\cc O$. In particular, if $\cc O=\cc O_{naive}$
then we have the motivic model structure on
$Pre^\Sigma(\text{Sm}/k)=\Mod\cc O_{naive}$ and we shall write write
$\sheff(k)$ to denote $\sheff\cc O_{naive}$.

}\end{defs}

We refer the reader to~\cite[Definition 5.7]{GP} for the notions of
{\it Nisnevich excisive\/} and {\it motivically excisive\/} spectral
categories. These basically mean that $\mathcal{O}$ converts
elementary distinguished squares to homotopy pushouts with respect
to the appropriate model structure.






Let $\text{AffSm}/k$ be the full subcategory of $\text{Sm}/k$ whose
objects are the smooth affine varieties. $\text{AffSm}/k$ gives rise
to a spectral category $\cc O_{\text{Aff}}$ whose objects are those
of $\text{AffSm}/k$ and morphisms spectra are defined as
   $$\cc O_{\text{Aff}}(X,Y):=\Hom_{\text{AffSm}/k}(X,Y)_+\wedge\bb S,$$
where $\bb S$ is the sphere spectrum and $X,Y\in \text{AffSm}/k$.

Recall that a sheaf $\cc F$ of abelian groups in the Nisnevich
topology on $\text{Sm}/k$ is {\it strictly $\bb A^1$-invariant\/} if
for any $X\in \text{Sm}/k$, the canonical morphism
   $$H^*_{\nis}(X,\cc F)\to H^*_{\nis}(X\times\bb A^1,\cc F)$$
is an isomorphism.

\begin{defs}{\rm
Let $\cc R$ be a ringoid  over $\text{Sm}/k$ together with the
structure functor $\rho:\text{Sm}/k\to\cc R$. We say that $\cc R$ is
a {\it $V$-ringoid\/} (``$V$" for Voevodsky) if

\begin{enumerate}
\item for any elementary distinguished square $Q$ the sequence of Nisnevich
sheaves associated to representable presheaves
   $$0\to\cc R_{\nis}(-,U')\to\cc R_{\nis}(-,U)\oplus\cc R_{\nis}(-,X')\to\cc R_{\nis}(-,X)\to 0$$
is exact;
\item there is a functor
   $$\boxtimes:\cc R\times \text{AffSm}/k\to\cc R$$
sending $(X,U)\in \text{Sm}/k\times \text{AffSm}/k$ to $X\times U\in
Sm/k$ and such that $1_X\boxtimes\alpha=\rho(1_X\times\alpha)$,
$(u+v)\boxtimes\alpha=u\boxtimes\alpha+v\boxtimes\alpha$ for all
$\alpha\in\Mor(\text{AffSm}/k)$ and $u,v\in\Mor(\cc R)$.
\item for any $\cc R$-presheaf of abelian groups $\cc F$, i.e. $\cc F$ is a contravariant
functor from $\cc R$ to abelian groups, the associated Nisnevich
sheaf $\cc F_{\nis}$ has a unique structure of a $\cc R$-presheaf
for which the canonical homomorphism $\cc F\to\cc F_{\nis}$ is a
homomorphism of $\cc R$-presheaves. Moreover, if $\cc F$ is homotopy
invariant then so is $\cc F_{\nis}$;
\end{enumerate}
We refer to $\cc R$ as a {\it strict $\bb A^1$-invariant
$V$-ringoid\/} if every $\bb A^1$-invariant Nisnevich $\cc R$-sheaf
is strictly $\bb A^1$-invariant.

}\end{defs}

We want to make several remarks regarding the definition. Condition
(1) implies the spectral category $\cc O_{\cc R}$ associated to the
ringoid $\cc R$ is Nisnevich excisive. Condition (2) implies that
for any $\cc R$-presheaf $\cc F$ and any affine scheme $U\in
\text{AffSm}/k$ the presheaf
   $$\underline\Hom(U,\cc F):=\cc F(-\times U)$$
is an $\cc R$-presheaf. Moreover, it is functorial in $U$.


\begin{defs}\label{vsp}{\rm
Let $(\cc O,\sigma)$ be a spectral category over $\text{Sm}/k$ in
the sense of Definition~\ref{spectral}(6). We say that $\cc O$ is a
{\it $V$-spectral category\/} if

\begin{enumerate}
\item $\cc O$ is connective and Nisnevich excisive;
\item there is a spectral functor
   $$\boxdot:\cc O\wedge\cc O_{\text{Aff}}\to\cc O$$
sending $(X,U)\in \text{Sm}/k\times \text{AffSm}/k$ to $X\times U\in
\text{Sm}/k$ and such that
$1_X\boxdot\alpha=\sigma(1_X\times\alpha)$ for all
$\alpha\in\Mor(\text{AffSm}/k)$;
\item $\pi_0\cc O$ is a $V$-ringoid such that the structure map $\rho:\text{Sm}/k\to\pi_0\cc O$
equals the composite map
   $$\text{Sm}/k\to\pi_0\cc O_{naive}\xrightarrow{\pi_0(\sigma)}\pi_0\cc O.$$
We also require the structure pairing $\boxtimes:\pi_0\cc O\times
\text{AffSm}/k\to\pi_0\cc O$ to be the composite functor
   $$\pi_0\cc O\times \text{AffSm}/k\xrightarrow{}\pi_0\cc O\times\pi_0\cc O_{\text{Aff}}\to\pi_0(\cc O\wedge\cc O_{\text{Aff}})
     \xrightarrow{\pi_0(\boxdot)}\pi_0\cc O.$$
\end{enumerate}
We refer to $\cc O$ as a {\it strict $V$-spectral category\/} if
$\pi_0\cc O$ is a strict $\bb A^1$-invariant $V$-ringoid.

}\end{defs}

Since the main category $D\cc O_-^{\text{eff}}(k)$ we shall work
with consists of bounded below $\cc O$-modules (see
section~\ref{dominus} for precise definitions), we assume $\cc O$ to
be connective in Definition~\ref{vsp}.

We note that if $\cc O$ is a $V$-spectral category, then for every
$\cc O$-module $M$ and any affine smooth scheme $U$, the presheaf of
symmetric spectra
   $$\underline{\Hom}(U,M):=M(-\times U)$$
is an $\cc O$-module. Moreover, $M(-\times U)$ is functorial in $U$.

\begin{lem}\label{pepe}
Every $V$-spectral category $\cc O$ is motivically excisive.
\end{lem}

\begin{proof}
Every $V$-spectral category is, by definition, Nisnevich excisive.
Since there is an action of affine smooth schemes on $\cc O$, the
fact that $\cc O$ is motivically excisive is proved similar
to~\cite[5.8]{GP}.
\end{proof}

Let $\cc O$ be a $V$-spectral category. Since it is both Nisnevich
and motivically excisive, it follows from~\cite[5.13]{GP} that the
pair of natural adjoint fuctors
   $$\xymatrix{{\Psi_*}:Pre^\Sigma(\text{Sm}/k)\ar@<0.5ex>[r]&\Mod\cc O:{\Psi^*}\ar@<0.5ex>[l]}$$
induces a Quillen pair for the Nisnevich local projective
(respectively motivic) model structures on $Pre^\Sigma(\text{Sm}/k)$
and $\Mod\cc O$. In particular, one has adjoint functors between
triangulated categories
   \begin{equation}\label{adjoint}
    {\Psi_*}:\shnis(k)\leftrightarrows\shnis\cc O:{\Psi^*}\quad\textrm{ and }\quad {\Psi_*}:\sheff(k)\leftrightarrows\sheff\cc O:{\Psi^*}.
   \end{equation}

\section{The triangulated category $D\cc O_-^{\text{\rm eff}}(k)$}\label{dominus}

Throughout this section we work with a strict $V$-spectral category
$\cc O$. We shall often work with simplicial $\cc O$-modules
$M[\bullet]$. The {\it realization\/} of $M[\bullet]$ is the $\cc
O$-module $|M|$ defined as the coend
   $$|M|=\Delta[\bullet]_+\wedge_{\Delta} M[\bullet]$$
of the functor $\Delta[\bullet]_+\wedge
M[\bullet]:\Delta\times\Delta^{\op}\to\Mod\cc O$. Here $\Delta[n]$
is the standard simplicial $n$-simplex.

Recall that the simplicial ring $k[\Delta]$ is defined as
   $$k[\Delta]_n=k[x_0,\ldots,x_n]/(x_0+\cdots+x_n-1).$$
By $\Delta^{\cdot}$ we denote the cosimplicial affine scheme
$\spec(k[\Delta])$. Let
   $$M\in\Mod\cc O\mapsto M_f\in\Mod\cc O$$
be a fibrant replacement functor in the Nisnevich local model
structure on $\Mod\cc O$. Given an $\cc O$-module $M$, we set
   $$C_*(M):=|\underline{\Hom}(\Delta^{\cdot},M_f)|.$$
Note that $C_*(M)$ is an $\cc O$-module and is functorial in $M$. If
we regard $M_f$ as a constant simplicial $\cc O$-module, the map of
cosimplicial schemes $\Delta^\cdot\to pt$ induces a map of $\cc
O$-modules
   $$M\to C_*(M).$$

\begin{lem}\label{spain}
The functor $C_*$ respects Nisnevich local weak equivalences. In
particular, it induces a triangulated endofunctor
   $$C_*:\shnis\cc O\to\shnis\cc O.$$
\end{lem}

\begin{proof}
Let $\alpha:L\to M$ be a Nisnevich local weak equivalence of $\cc
O$-modules. By~\cite[5.12]{GP} the forgetful functor $\Psi^*:\Mod\cc
O\to Pre^\Sigma(\text{Sm}/k)$ respects Nisnevich local weak
equivalences and Nisnevich local fibrant objects. It follows that
the fibrant replacement
   $$\alpha_f:L_f\to M_f$$
of $\alpha$ is a level equivalence of presheaves of ordinary
symmetric spectra, and hence so is each map
  $$\underline{\Hom}(\Delta^n,\alpha_f):\underline{\Hom}(\Delta^n,L_f)\to\underline{\Hom}(\Delta^n,M_f),\quad n\geq 0.$$
Since the realization functor respects level equivalences, our
assertion follows.
\end{proof}

One of advantages of strict $V$-spectral categories is that we can
construct an $\bb A^1$-local replacement of an $\cc O$-module $M$ in
two steps. We first take $C_*(M)$ and then its Nisnevich local
replacement $C_*(M)_f$.

\begin{thm}\label{italy}
The natural map $M\to C_*(M)_f$ is an $\bb A^1$-local replacement of
$M$ in the motivic model structure of $\cc O$-modules.
\end{thm}

\begin{proof}
The presheaves $\pi_i(C_*(M))$, $i\in\bb Z$, are homotopy invariant
and have $\pi_0\cc O$-transfers. Since $\cc O$ is a strict
$V$-spectral category then each Nisnevich sheaf
$\pi_i^{\nis}(C_*(M)_f)$ is strictly homotopy invariant and has
$\pi_0\cc O$-transfers. By~\cite[6.2.7]{Mor} $C_*(M)_f$ is $\bb
A^1$-local in the motivic model category structure on
$Pre^\Sigma(\text{Sm}/k)$. By~Lemma~\ref{pepe} $\cc O$ is
motivically excisive, hence~\cite[5.12]{GP} implies $C_*(M)_f$ is
$\bb A^1$-local in the motivic model category structure on $\Mod\cc
O$.

The map $M\to C_*(M)_f$ is the composite
   $$M\to M_f\to C_*(M)\to C_*(M)_f.$$
The left and right arrows are Nisnevich local trivial cofibrations.
The middle arrow is a level $\bb A^1$-weak equivalence in
$Pre^\Sigma(\text{Sm}/k)$ by~\cite[3.8]{MV}. By~Lemma~\ref{pepe}
$\cc O$ is motivically excisive, hence~\cite[5.12]{GP} implies the
middle arrow is an $\bb A^1$-weak equivalence in $\Mod\cc O$.
\end{proof}

\begin{defs}\label{omotive}{\rm
The $\cc O$-motive $M_{\cc O}(X)$ of a smooth algebraic variety
$X\in \text{Sm}/k$ is the $\cc O$-module $C_*(\cc O(-,X))$. We say
that an $\cc O$-module $M$ is {\it bounded below\/} if for $i\ll 0$
the Nisnevich sheaf $\pi_i^{\nis}(M)$ is zero. $M$ is {\it
$n$-connected\/} if $\pi_i^{\nis}(M)$ are trivial for $i\leq n$. $M$
is {\it connective\/} is it is $(-1)$-connected, i.e.,
$\pi_i^{\nis}(M)$ vanish in negative dimensions.

}\end{defs}

\begin{cor}\label{porto}
If an $\cc O$-module $M$ is bounded below (respectively
$n$-connected) then so is $C_*(M)$. In particular, the
$\cc O$-motive $M_{\cc O}(X)$ of any smooth algebraic variety
$X\in \text{Sm}/k$ is connective.
\end{cor}

\begin{proof}
This follows from the preceding theorem and Morel's Connectivity
Theorem~\cite{Mor}.
\end{proof}

Denote by $D\cc O_-(k)$ the full triangulated subcategory of
$\shnis\cc O$ of bounded below $\cc O$-modules. We also denote by
$D\cc O_-^{\text{eff}}(k)$ the full triangulated subcategory of
$D\cc O_-(k)$ of those $\cc O$-modules $M$ such that each Nisnevich
sheaf $\pi_i^{\nis}(M)$ is homotopy invariant. Note that for any
smooth algebraic variety $X\in \text{Sm}/k$ its $\cc O$-motive
$M_{\cc O}(X)$ belongs to $D\cc O_-^{\text{eff}}(k)$. To see this,
just apply Corollary \ref{porto} and Theorem \ref{neploho}(2) below.

The category $D\cc
O_-^{\text{eff}}(k)$ is an analog of Voevodsky's triangulated
category $DM_-^{\text{eff}}(k)$~\cite{Voe1}. Let $\cc O_{cor}$ be
the Eilenberg--Mac~Lane spectral category associated with the
ringoid $Cor$. We shall show below that $DM_-^{\text{eff}}(k)$ is
equivalent to $D\cc O_-^{\text{eff}}(k)$ if $\cc O=\cc O_{cor}$.

\begin{thm}\label{neploho}
Let $\cc O$ be a strict $V$-spectral category. Then the following
statements are true:

$(1)$ The kernel of $C_*$ is the full triangulated subcategory $\cc
T$ of $\shnis\cc O$ generated by the compact objects
   $$\cone(\cc O(-,X\times\bb A^1)\to\cc O(-,X)),\quad X\in \text{Sm}/k.$$
Moreover, the triangulated functor $C_*$ induces an equivalence of
triangulated categories
   $$\shnis\cc O/\cc T\lra{\sim}\sheff\cc O.$$

$(2)$ The functor
   $$C_*:D\cc O_-(k)\to D\cc O_-(k)$$
lands in $D\cc O_-^{\text{\rm eff}}(k)$. The kernel of $C_*$ is $\cc
T_-:=\cc T\cap D\cc O_-(k)$. Moreover, $C_*$ is left adjoint to the
inclusion functor
   $$i:D\cc O_-^{\text{\rm eff}}(k)\to D\cc O_-(k)$$
and $D\cc O_-^{\text{\rm eff}}(k)$ is equivalent to the quotient
category $D\cc O_-(k)/\cc T_-$.
\end{thm}

\begin{proof}
(1). The localization theory of compactly generated triangulated
categories implies the quotient category $\shnis\cc O/\cc T$ is
equivalent to the full triangulated subcategory
   $$\cc T^\bot=\{M\in\shnis\cc O\mid \Hom_{\shnis\cc O}(T,M)=0\textrm { for all $T\in\cc T$}\}.$$
Moreover,
   $$\cc T={}^\bot(\cc T^\bot)=\{X\in\shnis\cc O\mid \Hom_{\shnis\cc O}(X,M)=0\textrm { for all $M\in\cc T^\bot$}\}.$$

By construction, $\cc T^\bot$ can be identified up to natural
equivalence of triangulated categories with the full triangulated
subcategory of $\bb A^1$-local $\cc O$-modules. The latter
subcategory is naturally equivalent to $\sheff\cc O$, because the
motivic model structure on $\cc O$-modules is obtained from the
Nisnevich local model structure by Bousfield localization with
respect to the maps
   $$\cc O(-,X\times\bb A^1)\to\cc O(-,X),\quad X\in \text{Sm}/k.$$

Recall that a map $M\to N$ of $\cc O$-modules is a motivic
equivalence if and only if for any $\bb A^1$-local $\cc O$-module
$L$ the induced map
   $$\Hom_{\shnis\cc O}(N,L)\to\Hom_{\shnis\cc O}(M,L)$$
is an isomorphism. Given an $\cc O$-module $M$, the map $M\to
C_*(M)$ is a motivic equivalence by Theorem~\ref{italy}. If we fit
the arrow into a triangle in $\shnis\cc O$
   \begin{equation}\label{silva}
    X_M\to M\to C_*(M)\to X_M[1],
   \end{equation}
it will follow that $\Hom_{\shnis\cc O}(X_M,L)=0$ for all $L\in\cc
T^\bot$. We see that
for any
$\cc O$-module $M$
one has
$X_M\in{}^\bot(\cc T^\bot)=\cc T$.

If $C_*(M)\cong 0$ in $\shnis\cc O$, then $M\cong X_M \in \cc T$.
Thus, $M \in \cc T$ in this case.
On the other hand, if $M\in\cc T$ then $C_*(M)\in\cc T$, since
$X_M\in\cc T$ and $\cc T$ is a thick triangulated subcategory in
$\shnis\cc O$. On the other hand, Theorem~\ref{italy} implies
$C_*(M)\in\cc T^\bot$, and therefore $C_*(M)\in\cc T\cap\cc
T^\bot=0$. We conclude that $\cc T=\kr C_*$.

(2). For any $M\in\Mod\cc O$ the presheaves $\pi_i(C_*(M))$,
$i\in\bb Z$, are homotopy invariant and have $\pi_0\cc O$-transfers.
Since $\cc O$ is a strict $V$-spectral category then each Nisnevich
sheaf $\pi_i^{\nis}(C_*(M))$ is homotopy invariant. Therefore the
functor
   $$C_*:D\cc O_-(k)\to D\cc O_-(k)$$
lands in $D\cc O_-^{\text{eff}}(k)$. It follows from the first part
of the theorem that the kernel of $C_*$ is $\cc T_-:=\cc T\cap D\cc
O_-(k)$.

Let us prove that $D\cc O_-^{\text{eff}}(k)=\cc T^\bot\cap D\cc
O_-(k)$. Clearly, $\cc T^\bot\cap D\cc O_-(k)\subset D\cc
O_-^{\text{eff}}(k)$. Suppose $M\in D\cc O_-^{\text{eff}}(k)$. Then
$M_f\in D\cc O_-^{\text{eff}}(k)$. We have that $M_f$ is a fibrant
$\cc O$-module in the Nisnevich local model structure and each
$\pi_i^{\nis}(M_f)$ is a strictly homotopy invariant sheaf, because
$\cc O$ is a strict $V$-spectral category. By~\cite[6.2.7]{Mor}
$M_f$ is $\bb A^1$-local in the motivic model category structure on
$Pre^\Sigma(\text{Sm}/k)$. By~Lemma~\ref{pepe} $\cc O$ is
motivically excisive, hence~\cite[5.12]{GP} implies $M_f$ is $\bb
A^1$-local in the motivic model category structure on $\Mod\cc O$.
We see that $M\in \cc T^\bot\cap D\cc O_-(k)$.

Let $E\in D\cc O_-^{\text{eff}}(k)$ and $M\in D\cc O_-(k)$. Applying
the functor $\Hom_{D\cc O_-(k)}(-,E)$ to triangle~\eqref{silva}, one
gets
   $$\Hom_{D\cc O_-(k)}(M,E)\cong\Hom_{D\cc O_-(k)}(C_*(M),E)=\Hom_{D\cc O_-^{\text{eff}}(k)}(C_*(M),E).$$
Thus $C_*$ is left adjoint to the inclusion functor $i:D\cc
O_-^{\text{eff}}(k)\to D\cc O_-(k)$.

It remains to show that $D\cc O_-^{\text{eff}}(k)$ is equivalent to
the quotient category $D\cc O_-(k)/\cc T_-$. By the first part of
the theorem it is enough to prove that the natural functor
   $$D\cc O_-(k)/\cc T_-\to\shnis\cc O/\cc T$$
is fully faithful. Consider an arrow $M\xrightarrow{s}N$ in
$\shnis\cc O$, where $M\in D\cc O_-(k)$ and $s$ is such that
$\cone(s)\in\cc T$. There is a commutative diagram in $\shnis\cc O$
   $$\xymatrix{M\ar[d]_{u_M}\ar[r]^s\ar[d]&N\ar[d]^{u_N}\\
               C_*(M)\ar[r]^{C_*(s)}&C_*(N)}$$
in which cones of the vertical arrows are in $\cc T$. Since
$\cone(C_*(s))\cong C_*(\cone(s))=0$ in $\shnis\cc O$, we see that
$C_*(s)$ is an isomorphism in $\shnis\cc O$. Therefore $C_*(N)\in
D\cc O_-(k)$ and $\cone(u_N\circ s)\in\cc T_-$. By~\cite[9.1]{K96}
$D\cc O_-(k)/\cc T_-$ is a full subcategory of $\shnis\cc O/\cc T$.
\end{proof}

Suppose the field $k$ is perfect. Then~\cite{Voe} implies $\cc
O_{cor}$ is a strict $V$-spectral category. Recall that the
Voevodsky triangulated category of motives $DM_-^{\text{eff}}(k)$ is
the full triangulated subcategory of (cohomologically) bounded above
complexes of the derived category $D(ShTr)$ of Nisnevich sheaves
with transfers (see~\cite{SV1,Voe1}). The next result says that
$DM_-^{\text{eff}}(k)$ can be recovered from $D\cc
O_-^{\text{eff}}(k)$ if $\cc O=\cc O_{cor}$.

\begin{cor}\label{kongo}
Let $k$ be a perfect field and $\cc O=\cc O_{cor}$, then there is a
natural equivalence of triangulated categories
   $$D\cc O_-^{\text{\rm eff}}(k)\lra{\sim}DM_-^{\text{\rm eff}}(k).$$
\end{cor}

\begin{proof}
By~\cite[section~6]{GP} there is a natural equivalence of
triangulated categories $\shnis\cc O$ and $D(ShTr)$. Moreover, this
equivalence takes bounded below $\cc O$-modules to (cohomologically)
bounded above complexes. Restriction of the equivalence to $D\cc
O_-^{\text{eff}}(k)$ yields the desired equivalence between $D\cc
O_-^{\text{eff}}(k)$ and $DM_-^{\text{eff}}(k)$.
\end{proof}

To conclude the section, it is also worth to mention another way of
constructing a motivic fibrant replacement on $\cc O$-modules.
Namely, for any $M\in\Mod\cc O$ we set
   $$\wt C_*(M):=|d\mapsto(\underline{\Hom}(\Delta^d,M))_f|.$$
Clearly, $\wt C_*(M)$ is functorial in $M$. Observe that if $M$ is
Nisnevich local then $C_*(M)$ is zigzag level equivalent to $\wt
C_*(M)$, because $\underline{\Hom}(\Delta^d,M)$ and
$(\underline{\Hom}(\Delta^d,M))_f$ are Nisnevich local and the
arrows
   $$\underline{\Hom}(\Delta^d,M_f)\leftarrow\underline{\Hom}(\Delta^d,M)\to(\underline{\Hom}(\Delta^d,M))_f$$
are level weak equivalences.

\begin{prop}\label{ska}
The natural map $M\to\wt C_*(M)_f$ is an $\bb A^1$-local replacement
of $M$ in the motivic model structure of $\cc O$-modules.
\end{prop}

\begin{proof}
The map $M\to\wt C_*(M)_f$ is the composite
   $$M\to|d\mapsto\underline{\Hom}(\Delta^d,M)|\to|d\mapsto(\underline{\Hom}(\Delta^d,M))_f|\to\wt C_*(M)_f.$$
The left arrow is a level $\bb A^1$-weak equivalence in
$Pre^\Sigma(\text{Sm}/k)$ by~\cite[3.8]{MV}. The middle arrow is a
Nisnevich local weak equivalence, because it is the realization of a
simplicial Nisnevich local weak equivalence. The right arrow is
plainly a Nisnevich local weak equivalence as well.

The presheaves $\pi_i(|d\mapsto\underline{\Hom}(\Delta^d,M)|)$,
$i\in\bb Z$, are homotopy invariant and have $\pi_0\cc O$-transfers.
Since $\cc O$ is a strict $V$-spectral category then each Nisnevich
sheaf $\pi_i^{\nis}(\wt C_*(M)_f)$ is strictly homotopy invariant
and has $\pi_0\cc O$-transfers. By~\cite[6.2.7]{Mor} $C_*(M)_f$ is
$\bb A^1$-local in the motivic model category structure on
$Pre^\Sigma(\text{Sm}/k)$. By~Lemma~\ref{pepe} $\cc O$ is
motivically excisive, hence~\cite[5.12]{GP} implies the arrow of the
proposition is an $\bb A^1$-weak equivalence in $\Mod\cc O$.
\end{proof}

\section{The spectral category $\bb K$}\label{CategoriesA}

In this section the definition of the $V$-spectral category $\bb K$
is given. It is obtained by taking $K$-theory symmetric spectra
$K(\cc A(U,X))$ of certain additive categories $\cc A(U,X)$, $U,X\in
\text{Sm}/k$. To define these categories we need some preliminaries.

\begin{notn}\label{Supp}{\rm
Let $U,X\in \text{Sm}/k$. Define $\text{Supp}(U\times X/X)$ as the
set of all closed subsets in $U\times X$ of the form $A=\cup_{j\in
J} B_j$, where $J$ is a finite set and each $B_j$ is a closed
irreducible subset in $U\times X$ which is finite and surjective
over $U$. The empty subset $\empty$ in $U\times X$ is also regarded
as an element of $\text{Supp}(U\times X/X)$.

}\end{notn}

\begin{notn}\label{SubSch}{\rm
Given $U,X \in \text{Sm}/k$ and $A\in\text{Supp}(U\times X/X)$, let
$I_A\subset\mathcal O_{U \times X}$ be the ideal sheaf of the closed
set $A \subset U \times X$. Denote by $A_m$ the closed subscheme in
$U \times X$ of the form $(A,\mathcal O_{U \times X}/I^m_A)$. If
$m=0$, then $A_m$ is the empty subscheme. Define
$\text{SubSch}(U\times X/X)$ as the set of all closed subschemes in
$U \times X$ of the form $A_m$.

For any $Z\in\text{SubSch}(U\times X/X)$ we write $p^Z_U:Z\to U$ to
denote $p\circ i$, where $i:Z\hookrightarrow U\times X$ is the
closed embedding and $p: U \times X \to U$ is the projection. If
there is no likelihood of confusion we shall write $p_U$ instead of
$p^Z_U$, dropping $Z$ from notation.

Clearly, for any $Z \in \text{SubSch}(U\times X/X)$ the reduced
scheme $Z^{red}$, regarded as a closed subset of $U\times X$,
belongs to $\text{Supp}(U\times X/X)$.

}\end{notn}

\begin{notn}\label{BcircA}{\rm
Let $V,U,X \in \text{Sm}/k$. Let $A \in \text{Supp}(V\times U/U)$,
$B \in \text{Supp}(U\times X/X)$. Set
$$B \circ A = p_{VX}(V\times B \cap A \times X) \subset V\times X,$$
where $p_{VX}: V\times U\times X \to V\times X$ is the projection.
One can check that
$$B \circ A \in \text{Supp}(V\times X/X).$$

}\end{notn}

\begin{notn}\label{ZcircS}{\rm
Let $V,U,X \in \text{Sm}/k$. Let $S \in \text{SubSch}(V\times U/U)$,
$Z \in \text{Subsch}(U\times X/X)$. By~\ref{SubSch} one has $S^{red}
\in \text{Supp}(V\times U/U)$, $Z^{red} \in \text{Supp}(U\times
X/X)$. By~\ref{BcircA} one has $Z^{red} \circ S^{red} \in
\text{Supp}(V\times X/X)$. One can show that for some integer $k\gg
0$ there exists a scheme morphism
  $$\pi_k: T=S \times X \cap V\times Z\to (Z^{red} \circ S^{red})_k$$
such that $i_k \circ \pi_k= p_{VX} \circ i_T: T \to V\times X$. Here
$i_k: (Z^{red} \circ S^{red})_k \hookrightarrow V\times X $, $i_T: T
\hookrightarrow V\times U \times X$ are closed embeddings and
$p_{VX}: V \times U\times X \to V\times X$ is the projection.

If there exists $\pi_k$ satisfying the condition above then it is
unique. Moreover, for any $m > k$ one has $i^m_k \circ \pi_k=
\pi_m$, where $i^m_k: (Z^{red} \circ S^{red})_k \hookrightarrow
(Z^{red} \circ S^{red})_m$ is the closed embedding.

We shall often write $Z \circ S$ to denote $(Z^{red} \circ
S^{red})_k$, provided that there exists the required $\pi_k$. In
this case we shall also write $\pi$ to denote $\pi_k: T \to (Z \circ
S)$.

}\end{notn}

\begin{defs}[of additive categories $\cc A(U,X)$]\label{ccAobjects}{\rm
For any $U,X\in \text{Sm}/k$ we define objects of $\cc A(U,X)$ as
equivalence classes of the triples
   $$(n,Z,\phi:p_{U,*}(\cc O_Z)\to M_n(\cc O_U)),$$
where $n$ is a nonnegative integer, $Z \in \text{SubSch}(U\times
X/X)$ and $\phi$ is a non-unital homomorphism of sheaves of $\cc
O_U$-algebras. Let $p(\phi)$ be the idempotent $\phi(1)\in
M_n(\Gamma(U,\cc O_U))$, then $P(\phi)=\im(p(\phi))$ can be regarded
as a $p_{U,*}(\cc O_Z)$-module by means of $\phi$.

By definition, two triples $(n,Z,\phi)$, $(n',Z',\phi')$ are
equivalent if $n=n'$ and there is a triple $(n'',Z'',\phi'')$ such
that $n=n'=n''$, $Z,Z'\subset Z''$ are closed subschemes in $Z''$,
and the diagrams
   \begin{equation*}
    \xymatrix{p_{U,*}(\cc O_Z)\ar[rr]^\phi&&M_n(\cc O_U)&p_{U,*}(\cc O_{Z'})\ar[rr]^{\phi'}&&M_n(\cc O_U)\\
              &p_{U,*}(\cc O_{Z''})\ar[ul]^{can}\ar[ur]_{\phi''}&&&p_{U,*}(\cc O_{Z''})\ar[ul]^{can}\ar[ur]_{\phi''}}
   \end{equation*}
are commutative. We shall often denote an equivalence class for the
triples by $\Phi$. Though $Z$ is not uniquely defined by $\Phi$,
nevertheless we shall also refer to $Z\subset U\times X$ as the {\it
support\/} of $\Phi$.

Given $\Phi,\Phi'\in\cc A(U,X)$ we first equalize supports $Z,Z'$ of
the objects $\Phi,\Phi'$ and then set
   $$\Hom_{\cc A(U,X)}(\Phi,\Phi')=\Hom_{p_{U,*}(\cc O_Z)}(P(\phi),P(\phi')),$$
where the right hand side is an Abelian group in the usual way.
Given any three objects $\Phi,\Phi', \Phi'' \in\cc A(U,X)$ a
composition law
   $$\Hom_{\cc A(U,X)}(\Phi,\Phi') \circ \Hom_{\cc A(U,X)}(\Phi',\Phi'') \to \Hom_{\cc A(U,X)}(\Phi,\Phi'')$$
is defined in the obvious way. This makes therefore $\cc A(U,X)$ an
additive category. The zero object is the equivalence class of the
triple $(0, \emptyset, 0)$. By definition,\footnotesize
   $$\Phi_1\oplus\Phi_2=(n_1+n_2,Z_1\cup Z_2,p_{U,*}(\cc O_{Z_1\cup Z_2})\to p_{U,*}(\cc O_{Z_1})\times p_{U,*}(\cc O_{Z_2})
     \to M_{n_1}(\cc O_U))\times M_{n_2}(\cc O_U))\hookrightarrow M_{n_1+n_2}(\cc O_U)).$$
\normalsize Clearly, $P(\phi_1\oplus\phi_2)\cong P(\phi_1)\oplus
P(\phi_2)$. Definition of the additive category $\cc A(U,X)$ is
finished.

}\end{defs}

We now want to construct a bilinear pairing
   \begin{equation}\label{luu}
    \cc A(V,U)\times\cc A(U,X)\lra{\circ}\cc A(V,X),\quad U,V,X\in \text{Sm}/k.
   \end{equation}
First, define it on objects. Namely,
   $$((n_1,Z_1,\phi_1),(n_2,Z_2,\phi_2))\longmapsto(n_1n_2,Z_2\circ Z_1,\phi_2\circ\phi_1),$$
where
$Z_2\circ Z_1 \in \text{SubSch}(V\times X/X)$ is a closed subscheme of $V\times X$ defined in
Notation \ref{ZcircS}.
The nonunital homomorphism
$\phi_2\circ\phi_1: p^{(Z_2 \circ Z_1)}_{V,*}(\cc O_{Z_2\circ Z_1}) \to M_{n_2n_1}(\cc O_V)$
is given by the composition
   \begin{equation}\label{composition}
    \xymatrix{&&M_{n_2}(M_{n_1}(\cc O_V))\ar[r]^{L}_{\cong}&M_{n_2n_1}(\cc O_V)\\
              q_{V,*}(\cc O_{Z_1\times_U Z_2})=p_{V,*}(p_{Z_1,*}(\cc O_{Z_1\times_U Z_2}))\ar[rr]^(.63){p_{V,*}(\phi_{2,Z_1})}
              &&M_{n_2}(p_{V,*}(\cc O_{Z_1}))\ar[u]_{M_{n_2}(\phi_1)}\\
              p^{(Z_2 \circ Z_1)}_{V,*}(\pi_*(\cc O_{Z_1\times_U Z_2}))\ar@{=}[u]\\
              p^{(Z_2 \circ Z_1)}_{V,*}(\cc O_{Z_2\circ Z_1})\ar[u]^{p^{Z_2 \circ Z_1}_{V,*}(\pi^*)=can'}}
   \end{equation}
where $L$ is a canonical isomorphism obtained by inserting
$(n_1,n_1)$-matrices into entries of a $(n_2,n_2)$-matrix, the
diagrams
   $$\xymatrix{Z_1\times_UZ_2\ar@/_1pc/[dd]_{q_V}\ar[d]^{p_{Z_1}}\ar[r]&Z_2\ar[d]_{p_U}\ar[r]&X&Z_1\times_UZ_2\ar[dd]_{q_V}\ar[dr]\ar[drr]^\pi\\
               Z_1\ar[d]^{p_V}\ar[r]_{r}&U&&&V\times X\ar[dl]&\ar@{ >->}[l]Z_2\circ Z_1\ar[dll]^{p^{Z_2 \circ Z_1}_V}\\
               V&&&V}$$
are commutative, and
$\pi^{*}:\cc O_{Z_2\circ Z_1}\to\pi_*(\cc O_{Z_1\times_UZ_2})$
is induced by the scheme morphism
$\pi: Z_1\times_UZ_2 \to Z_2\circ Z_1$
from Notation \ref{ZcircS}.
Finally,
$\phi_{2,Z_1}: p_{Z_1,*}(\cc O_{Z_1\times_U Z_2}) \to M_{n_2}(\cc O_{Z_1})$
is defined as a unique non-unital homomorphism
of sheaves of $\cc O_{Z_1}$-algebras such that
for any open affine $U' \subset U$ and any open affine $Z'_1 \subset
Z_1$ with $r(Z'_1) \subset U'$ and $Z'_2=p^{-1}_U(U')$ the value of
$\phi_{2,Z_1}$ on $Z'_1$ coincides with the non-unital homomorphism
of $k[Z'_1]$-algebras
   $$k[Z'_1]\otimes_{k[U']} k[Z'_2] \xrightarrow{\id\otimes\phi_2}k[Z'_1]\otimes_{k[U']} M_{n_2}(k[U'])
     \xrightarrow{a \otimes \beta \mapsto a\cdot r^*(\beta)} M_{n_2}(k[Z'_1]).$$
For a future use set $p(\phi_{2,Z_1})=\phi_{2,Z_1}(1) \in
M_{n_2}(\Gamma(Z_1, \cc O_{Z_1}))$ and
$P(\phi_{2,Z_1})=\im[p(\phi_{2,Z_1}): \cc O^{n_2}_{Z_1} \to \cc
O^{n_2}_{Z_1}]$.

In order to define pairing~\eqref{luu} on morphisms, we need some
preparations. Let $\Phi_1\in\cc A(V,U)$ and $\Phi_2\in\cc A(U,X)$.
Consider the diagram
   $$\xymatrix{p_{V,*}(\cc O^{n_2}_{Z_1})\otimes_{p_{V,*}(\cc O_{Z_1})}P(\phi_1)
               \ar[r]^(.7){can}_(.7){\cong}&P(\phi_1)^{n_2}\ar@{ >->}[r]^{i_1}&(\cc O^{n_1}_V)^{n_2}
               \ar[r]^{\ell}_\cong&\cc O^{n_2n_1}_V\ar@<1ex>[d]^{p(\phi_2\circ\phi_1)}\\
               p_{V,*}(P(\phi_{2,Z_1}))\otimes_{p_{V,*}(\cc O_{Z_1})}P(\phi_1)\ar@{ >->}[u]^{p_{V,*}(i_{2,Z_1})\otimes\id}\ar[rrr]^{\sigma_{12}}
               &&&P(\phi_2\circ\phi_1),\ar@<1ex>[u]^{i(\phi_2\circ\phi_1)}}$$
where $\sigma_{12}=p(\phi_2\circ\phi_1)\circ\ell\circ i_1\circ
can\circ(i_{2,Z_1}\otimes\id)$ (here $\ell (e_{i,j})=
e_{i+(j-1)n_1}$). It is worth to note that the isomorphism $\ell$
induces an $\cc O_V$-algebra isomorphism $M_{n_2}(M_{n_1}(\cc O_V))
\cong M_{n_2n_1}(\cc O_V)$ which coincides with the canonical
isomorphism $L$ obtained by inserting $(n_1,n_1)$-matrices into
entries of a $(n_2,n_2)$-matrix.

\begin{defs}\label{TwoModules}{\rm
An $p^{(Z_2 \circ Z_1)}_{V,*}(\cc O_{Z_2\circ Z_1})$-module
structure on $p_{V,*}(P(\phi_{2,Z_1}))\otimes_{p_{V,*}(\cc
O_{Z_1})}P(\phi_1)$ is defined as follows. For any open $V^0 \subset
V$, $f \in \Gamma(V^0, p^{(Z_2 \circ Z_1)}_{V,*}(\mathcal
O_{Z_2\circ Z_1}))$, $m_1 \in \Gamma(V^0, P(\phi_1))$, and $m_2 \in
\Gamma(V^0, p_{V,*}(P(\phi_{2,Z_1})))$
set
   $$f(m_2\otimes m_1)=((p_{V,*}(\phi_{2,Z_1})\circ can')(f))(m_2) \otimes m_1.$$
An $p^{(Z_2 \circ Z_1)}_{V,*}(\cc O_{Z_2\circ Z_1})$-module
structure on $P(\phi_2\circ\phi_1)$ is defined as follows. For any
open $V^0 \subset V$, $f \in \Gamma(V^0, p^{(Z_2 \circ
Z_1)}_{V,*}(\mathcal O_{Z_2\circ Z_1}))$, and $m \in \Gamma(V^0,
P(\phi_2\circ\phi_1))$ set
   $$fm=((\phi_2 \circ \phi_1)(f))(m).$$
In particular,
   $$1 \cdot m=((\phi_2 \circ \phi_1)(1))(m)=p(\phi_2\circ\phi_1)(m) =m,$$
because $m\in\im(p(\phi_2\circ\phi_1))$.

}\end{defs}

\begin{lem}
The map $\sigma_{12}$ is an isomorphism of $\cc O_V$-modules and,
moreover, an isomorphism of the
$p^{(Z_2 \circ Z_1)}_{V,*}(\cc O_{Z_2\circ Z_1})$-modules.
\end{lem}

Let $\alpha_1:\Phi_1\to\Phi_1'$ and $\alpha_2:\Phi_2\to\Phi_2'$ be
morphism in $\cc A(V,U)$ and $\cc A(U,X)$ respectively. We set
   \begin{equation}\label{oleg}
    \alpha_2\odot\alpha_1=\sigma_{12}'\circ(\alpha_2\otimes\alpha_1)\circ\sigma_{12}^{-1}:P(\phi_2\circ\phi_1)
    \to P(\phi_2'\circ\phi_1').
   \end{equation}
The definition of pairing~\eqref{luu} is finished. It is defined on
objects above and on morphisms by formula~\eqref{oleg}.

\begin{lem}\label{hoho}
The functor $\cc A(V,U)\times\cc A(U,X)\lra{\circ}\cc A(V,X)$ is
bilinear for all $U,V,X\in \text{Sm}/k$.
\end{lem}

For any $X\in \text{Sm}/k$ we define an object $\id _X\in\Ob\cc
A(X,X)$ by
   $$\id_X=(1,\Delta_X,\id:\cc O_X\to\cc O_X).$$

\begin{lem}\label{hohoho}
For any $U,X\in \text{Sm}/k$ the functors $\{\id_U\}\times\cc
A(U,X)\lra{\circ}\cc A(U,X)$ and $\cc
A(U,X)\times\{\id_X\}\lra{\circ}\cc A(U,X)$ are identities on $\cc
A(U,X)$.
\end{lem}

\begin{lem}\label{nono}
For any $U,V,W,X\in \text{Sm}/k$ and any $\Phi_1\in\cc
A(W,V),\Phi_2\in\cc A(V,U),\Phi_3\in\cc A(U,X)$ the following
statements are true:
\begin{enumerate}
\item $\Phi_3\circ(\Phi_2\circ\Phi_1)=(\Phi_3\circ\Phi_2)\circ\Phi_1\in\Ob\cc
A(W,X)$;

\item
$p(\phi_3\circ(\phi_2\circ\phi_1))=p((\phi_3\circ\phi_2)\circ\phi_1)$
and
$P(\phi_3\circ(\phi_2\circ\phi_1))=P((\phi_3\circ\phi_2)\circ\phi_1)$;

\item suppose $\alpha_i:\Phi_i\to\Phi_i'$ are morphisms $(i=1,2,3)$, then
$\alpha_3\odot(\alpha_2\odot\alpha_1)=(\alpha_3\odot\alpha_2)\odot\alpha_1\in\Hom_{\cc
A(W,X)}(P(\phi_3\circ(\phi_2\circ\phi_1)),P((\phi_3'\circ\phi_2')\circ\phi_1'))$.
\end{enumerate}
\end{lem}

\begin{prop}\label{nonono}
For any $U,V,W,X\in \text{Sm}/k$ the diagram of
functors\footnotesize
   $$\xymatrix{&(\cc A(W,V)\times\cc A(V,U))\times\cc A(U,X)\ar[r]^(.585){\circ\times\id}&\cc A(W,U)\times\cc A(U,X)\ar[dd]_{\circ}\\
               \cc A(W,V)\times(\cc A(V,U)\times\cc A(U,X))\ar[ur]^{\cong}\ar[dr]_{\id\times\circ}\\
               &\cc A(W,V)\times\cc A(V,X)\ar[r]^{\circ}&\cc A(W,X)}$$
\normalsize is strictly commutative.
\end{prop}

\begin{lem}\label{hirsch}
Pairings~\eqref{luu} together with $\{\id_X\}_{X\in \text{Sm}/k}$
determine a category $\cc A$ on $\text{Sm}/k$ which is also enriched
over additive categories. Moreover, the rules $X\mapsto X$ and
$f\mapsto\Phi_f=(1,\Gamma_f,\id:\cc O_U\to\cc O_U)$ give a functor
$\sigma:\text{Sm}/k\to\cc A$.
\end{lem}

The following notation will be useful later.

\begin{notn}\label{f^*andg_*}{\rm
Let $X,X',Y \in \text{Sm}/k$ and $f:X' \to X$ be a morphism in
$\text{Sm}/k$. Define a functor $f^*: \cc A(X,Y) \to \cc A(X',Y)$ as
the additive functor
   $$\cc A(X,Y) \xrightarrow{- \circ \sigma(f)} \cc A(X',Y).$$
More precisely, $f^*(\Phi)=\Phi \circ \sigma(f)$ and
$f^*(\alpha)=\alpha \odot\id_{\sigma(f)}$.

Let $X,Y,Y' \in \text{Sm}/k$ and $g:Y \to Y'$ be a morphism in
$\text{Sm}/k$. Define a functor $g_*: \cc A(X,Y) \to \cc A(X,Y')$ as
the additive functor
$$\cc A(X,Y) \xrightarrow{\sigma(g) \circ -} \cc A(X,Y').$$
Namely, $g_*(\Phi)=\sigma(g) \circ \Phi$ and
$g_*(\alpha)=\id_{\sigma(g)} \odot \alpha$.

}\end{notn}

Using this notation and Proposition~\ref{nonono}, one has the
following

\begin{cor}\label{propertiesf*andg*}
Let $f: X' \to X$ and $g: Y \to Y'$ be morphisms in $\text{Sm}/k$.
Then $f^* \circ g_* = g_* \circ f^*: \cc A(X,Y) \to \cc A(X',Y')$.
If $f': X'' \to X'$ is a map in $\text{Sm}/k$ then $(f \circ f')^*=
(f')^* \circ f^*: \cc A(X,Y) \to \cc A(X'',Y)$. Also, for any map
$g': Y' \to Y''$ in $\text{Sm}/k$ one has $(g' \circ g)_*= (g')_*
\circ g_*: \cc A(X,Y) \to \cc A(X,Y'')$.
\end{cor}

By~\cite[\S6.1]{GH} for an additive category $\cc C$, one can define
the structure of a symmetric spectrum on the Waldhausen $K$-theory
spectrum $K(\cc C)$. By definition,
   $$K(\cc C)_n=|\Ob S.{}^Q\cc C|,\quad Q=\{1,\ldots,n\}.$$
Moreover, strictly associative bilinear pairings of additive
categories induce strictly associative pairings of their $K$-theory
symmetric spectra (see~\cite[\S6.1]{GH}). The spectrum $K(\mathcal
C)$ is connective as any Waldhausen $K$-theory spectrum.

\begin{notn}\label{kux}{\rm
For any $U,X\in \text{Sm}/k$, we denote by $\bb K(U,X)$ the
Waldhausen $K$-theory symmetric spectrum $K(\cc A(U,X))$, where $\cc
A(U,X)$ is the additive category in the sense of
Definition~\ref{ccAobjects}.

}\end{notn}

Pairing~\eqref{luu} yields a pairing of symmetric spectra
   \begin{equation}\label{forel}
    \bb K(V,U)\wedge\bb K(U,X)\to\bb K(V,X).
   \end{equation}
Proposition~\ref{nonono} implies that~\eqref{forel} is a strictly
associative pairing. Moreover, for any $X\in \text{Sm}/k$ there is a
map $\mathbf 1:S\to\bb K(X,X)$ which is subject to the unit
coherence law (see~\cite[section~6.1]{GH}). Note that $\mathbf
1_0:S^0\to\bb K(X,X)_0$ is the map which sends the basepoint to the
null object and the non-basepoint to the unit object $\id_X$.

Thus we get the following

\begin{thm}\label{semga}
The triple $(\bb K,\wedge,\mathbf 1)$ determines a spectral
category. Moreover, the functor $\sigma:\text{Sm}/k\to\cc A$ of
Lemma~\ref{hirsch} gives a spectral functor
   $$\sigma:\cc O_{naive}\to\bb K$$
between spectral categories such that the pair $(\bb K,\sigma)$ is a
spectral category over $\text{Sm}/k$ in the sense of
Definition~\ref{spectral}(6).
\end{thm}

We now want to define a spectral functor
   \begin{equation*}
    \boxdot:\bb K\wedge\cc O_{naive}\to\bb K.
   \end{equation*}
It is in fact determined by additive functors
   $$f^{\star}:\cc A(X,X')\to\cc A(X\times U,X'\times U'),\quad f:U\to U'\in\Mor(\text{Sm}/k),$$
satisfying certain reasonable properties mentioned below. If
   $$(n,Z,\phi:p^Z_{X,*}(\cc O_Z)\to M_n(\cc O_X))$$
is a representative for $\Phi\in\cc A(X,X')$, then $f^{\star}(\Phi)$ is
represented by the triple
   $$(n,Z\times \Gamma_f,\phi\boxtimes\id_U:(p^Z_X\times\id)_*(\cc O_{Z\times\Gamma_f})\to M_n(\cc O_{X\times U})).$$
Here
$\phi\boxtimes\id_U$
is a unique non-unital homomorphism of sheaves of $\cc O_{X\times U}$-algebras such that
for any affine open subsets
$X_0 \subset X$, $U_0 \subset U$
and for
$Z_0=(p^Z_X)^{-1}(X_0) \subset Z$
the value of
$\phi\boxtimes\id_U$ on $X_0 \times U_0$ is the following non-unital homomorphism of
$k[Z_0 \times U_0]$-algebras:
$$(\phi\boxtimes\id_U)(a\boxtimes b):=(q_{X,0})^*(\phi(a))\cdot (q_{U,0})^*(b) \in M_n(k[X_0\times U_0]),$$
where
$q_{X,0}: X^0\times U^0 \to X^0$ and $q_{U,0}: X^0\times U^0 \to U^0$ are the projections.


To define $f^{\star}$ on morphisms, we note that the canonical morphism
$$adj: q_X^*(P(\phi)) \xrightarrow{q^*_X(i_{P(\phi)})} q^*_X(\cc O^n_X) \xrightarrow{can} \cc O^n_{X\times U} \xrightarrow{p(f^{\star}(\Phi))} P(f^{\star}(\Phi))$$
is an isomorphism.
Given a morphism
$\alpha:\Phi\to\Phi'$ in $\cc A(X,X')$, we set
   $$f^{\star}(\alpha)=adj^{\prime}\circ q_X^*(\alpha)\circ adj^{-1}:P(f^{\star}(\Phi))\to P(f^{\star}(\Phi')).$$
Clearly, $f^{\star}$ is an additive functor.

\begin{prop}\label{uha}
Let $f_1:U\to U'$, $f_2:U'\to U''$ be two maps in $\text{Sm}/k$,
$\Phi_1,\Phi_1'\in\Ob\cc A(X,X')$, $\Phi_2,\Phi_2'\in\Ob\cc
A(X',X'')$, let $\alpha_1:\Phi_1\to\Phi_1'$ be a morphism in $\cc
A(X,X')$ and let $\alpha_2:\Phi_2\to\Phi_2'$ be a morphism in $\cc
A(X',X'')$. Then

\begin{enumerate}
\item $(f_2\circ f_1)^{\star}(\Phi_2\circ\Phi_1)=f_2^{\star}(\Phi_2)\circ
f_1^{\star}(\Phi_1)$;
\item $(f_2\circ f_1)^{\star}(\alpha_2\odot\alpha_1)=f_2^{\star}(\alpha_2)\odot
f_1^{\star}(\alpha_1)$.
\end{enumerate}
\end{prop}

\begin{cor}\label{bahn}
Under the assumptions of Proposition~\ref{uha} the diagram of
functors
   $$\xymatrix{\cc A(X,X')\times\cc A(X',X'')\ar[d]_{f_1^{\star}\times f_2^{\star}}\ar[r]^\circ&\cc A(X,X'')\ar[d]^{(f_2\circ f_1)^{\star}}\\
               \cc A(X\times U,X'\times U')\times\cc A(X'\times U',X''\times U'')\ar[r]^(.65)\circ&\cc A(X\times U,X''\times U'')}$$
is commutative.
\end{cor}

\begin{cor}\label{tyty}
We have a spectral functor
   $$\boxdot:\bb K\wedge\cc O_{naive}\to\bb K$$
such that $(X,U)\in \text{Sm}/k\times \text{Sm}/k$ is mapped to
$X\times U\in \text{Sm}/k$. Moreover, for any morphism $h:X\to X'$
in $\text{Sm}/k$, regarded as the object $\sigma(h)$ of $\cc
A(X,X')$, one has
   $$f^{\star}(\sigma(h))=\sigma(f\times h)\in\Ob\cc A(X\times U,X'\times U')$$
for every morphism of $k$-smooth schemes $f:U\to U'$.
\end{cor}

In what follows we shall denote by $\bb K_0$ the ringoid $\pi_0(\bb
K)$.

\begin{thm}[Knizel~\cite{Kni}]\label{knizel}
For any $\bb K_0$-presheaf of abelian groups $\cc F$, i.e. $\cc F$
is a contravariant functor from $\bb K_0$ to abelian groups, the
associated Nisnevich sheaf $\cc F_{\nis}$ has a unique structure of
a $\bb K_0$-presheaf for which the canonical homomorphism $\cc
F\to\cc F_{\nis}$ is a homomorphism of $\bb K_0$-presheaves. If $\cc
F$ is homotopy invariant then so is $\cc F_{\nis}$. Moreover, if the
field $k$ is perfect then every $\bb A^1$-invariant Nisnevich $\bb
K_0$-sheaf is strictly $\bb A^1$-invariant.
\end{thm}

\begin{rem}{\rm
Although the category $\cc A(X,Y)$ is different from the category of
bimodules $\cc P(X,Y)$ (see Appendix for the definition of $\cc
P(X,Y)$), the proof of the preceding theorem is in spirit similar to
the proof of the same fact for $K_0$-presheaves obtained by
Walker~\cite{Wlk}.

}\end{rem}

\begin{prop}\label{psps}
$\bb K_0$ is a $V$-ringoid. If the field $k$ is perfect then it is
also a strict $\bb A^1$-invariant $V$-ringoid.
\end{prop}

\begin{proof}
The proof of~\cite[5.9]{GP} shows that for any elementary
distinguished square the sequence of Nisnevich sheaves associated to
representable presheaves
   $$0\to\bb K_{0,\nis}(-,U')\to\bb K_{0,\nis}(-,U)\oplus\bb K_{0,\nis}(-,X')\to\bb K_{0,\nis}(-,X)\to 0$$
is exact.

Let $\rho:\text{Sm}/k\to\bb K_0$ be the composite functor
   \begin{equation}\label{zamok}
    \text{Sm}/k\to\pi_0\cc O_{naive}\xrightarrow{\pi_0(\sigma)}\pi_0(\bb K)=\bb K_0,
   \end{equation}
where $\sigma:\cc O_{naive}\to\bb K$ is the spectral functor
constructed in Theorem~\ref{semga}. Also, let a functor
$\boxtimes:\bb K_0\times \text{Sm}/k\to\bb K_0$ be the composite
functor
   \begin{equation}\label{nemo}
    \bb K_0\times \text{Sm}/k\xrightarrow{}\bb K_0\times \pi_0\cc O_{naive}\to\pi_0(\bb K\wedge\cc O_{naive})
     \xrightarrow{\pi_0(\boxdot)}\bb K_0,
   \end{equation}
where $\boxdot:\bb K\wedge\cc O_{naive}\to\bb K$ is the spectral
functor constructed in Corollary~\ref{tyty}. Then we have that
$\id_X\boxtimes f=\rho(\id_X\times f)$, $(u+v)\boxtimes f=u\boxtimes
f+v\boxtimes f$ for all $u,v\in\Mor(\bb K_0)$ and
$f\in\Mor(\text{Sm}/k)$.

Theorem~\ref{knizel} now implies $\bb K_0$ is a $V$-ringoid. It also
follows from Theorem~\ref{knizel} that it is a strict $\bb
A^1$-invariant $V$-ringoid over perfect fields.
\end{proof}

We are now in a position to prove the main result of the section.

\begin{thm}\label{ooops}
The spectral category $\bb K$ together with the spectral functor
$\sigma:\cc O_{naive}\to\bb K$ of Theorem~\ref{semga} is a
$V$-spectral category in the sense of Definition~\ref{vsp}. If the
field $k$ is perfect then it is also a strict $V$-spectral category.
\end{thm}

\begin{proof}
$\bb K$ is connective by construction. It is proved similar
to~\cite[5.9]{GP} that $\bb K$ is Nisnevich excisive. The structure
spectral functor
   $$\sigma:\cc O_{naive}\to\bb K$$
is constructed in Theorem~\ref{semga}.

It follows from Corollary~\ref{tyty} that there is a spectral
functor
   $$\boxdot:\bb K\wedge\cc O_{naive}\to\bb K$$
sending $(X,U)\in \text{Sm}/k\times \text{Sm}/k$ to $X\times U\in
\text{Sm}/k$ and such that $\id_X\boxdot f=\sigma(\id_X\times f)$
for all $f\in\Mor(\text{Sm}/k)$. Proposition~\ref{psps} implies the
ringoid $\bb K_0$ together with structure functors~\eqref{zamok}
and~\eqref{nemo} is a $V$-ringoid which is strict $\bb
A^1$-invariant whenever the base field $k$ is perfect.
\end{proof}

We are now able to introduce the triangulated category of
$K$-motives.

\begin{defs}\label{prekrasno}{\rm
Suppose $k$ is a perfect field. The {\it triangulated category of
$K$-motives\/} $DK_-^{\text{eff}}(k)$ is the triangulated category
$D\cc O_-^{\text{eff}}(k)$ constructed in Section~\ref{dominus}
associated to the strict $V$-spectral category $\cc O=\bb K$ of
Theorem~\ref{ooops}.

}\end{defs}

To conclude the section, we discuss further useful properties of
categories $\cc A(U,X)$-s.

\begin{prop}\label{AlphaUalphaV}
Under Notation \ref{f^*andg_*} and the notation of Lemma
\ref{hirsch} and the notation which are just above Proposition
\ref{uha} for any $X,Y \in \text{Sm}/k$ and any morphism $f:U\to V$
in $\text{Sm}/k$ the following square of additive functors is
strictly commutative
   $$\xymatrix{\cc A(X\times V,Y\times V)\ar[rrr]^{(1_X\times f)^*}&&&\cc A(X\times U,Y\times V)\\
               \cc A(X,Y)\ar[u]_{\id^{\star}_V}\ar[rrr]^{\id^{\star}_U}&&&\cc A(X\times U,Y\times U).
               \ar[u]_{(1_Y \times f)_*}}$$
\end{prop}

\begin{notn}\label{CategWithAut}{\rm
For every $X\in \text{Sm}/k,Y\in \text{Sm}/k$ and $n>0$, denote by
$\cc A(X,Y)(\bb G_m^{\times n})$ the category whose objects are the
tuples $(\Phi,\theta_1,\ldots,\theta_n)$, where $\Phi \in\cc A(X,Y)$
and $(\theta_1,\ldots,\theta_n)$ are commuting automorphisms of
$\Phi$. Morphisms from $(\Phi,\theta_1,\ldots,\theta_n)$ to
$(\Phi',\theta'_1,\ldots,\theta'_n)$ are given by morphisms $\alpha:
\Phi \to \Phi'$ in $\cc A(X,Y)$ such that
$\alpha\circ\theta_i=\theta'_i\circ\alpha$ for every $i=1,\ldots,n$.

Using Notation \ref{f^*andg_*} for a morphism $f: X' \to X$ in
$\text{Sm}/k$, define an additive functor
   $$f^*_n: \cc A(X,Y)(\bb G_m^{\times n}) \to \cc A(X',Y)(\bb G_m^{\times n})$$
as follows:
$f^*_n(\Phi,\theta_1,\ldots,\theta_n)=(f^*(\Phi),f^*(\theta_1),
\ldots, f^*(\theta_n))$ on objects and $f^*_n(\alpha)=f^*(\alpha)$
on morphisms.

Using Notation \ref{f^*andg_*} for a morphism $g: Y \to Y'$ in
$\text{Sm}/k$, define an additive functor
$$g_{*,n}: \cc A(X,Y)(\bb G_m^{\times n}) \to \cc A(X,Y')(\bb G_m^{\times n})$$
as follows: $g_{*,n}(\Phi,\theta_1,\ldots,\theta_n)=(g_*(\Phi),
g_*(\theta_1) \ldots, g_*(\theta_n))$ on objects and
$g_{n,*}(\alpha)=g_*(\alpha)$ on morphisms.

}\end{notn}

\begin{defs}\label{CatFunctor}{\rm
Given $X\in \text{Sm}/k,Y\in \text{Sm}/k$ and $n>0$, we define an
additive functor
   $$\rho_{X,Y,n}: \cc A(X,Y\times\bb G_m^{\times n}) \to \cc A(X,Y)(\bb G_m^{\times n})$$
by using the functor $(pr_Y)_*: \cc A(X,Y \times\bb G_m^{\times n})
\to A(X,Y)$ from Notation~\ref{f^*andg_*} as follows. Given an
object $\Phi \in \cc A(X,Y \times\bb G_m^{\times n})$ and its
representative
   $$(n,Z,\varphi: p_{X,*}(\cc O_Z) \to M_n(\cc O_X)),$$
we have $n$ automorphisms $[t_i]$'s of $\Phi$ of the form $m \mapsto
\varphi(t_i|_{Z})m$, where each $t_i=p^*_i(t) \in \Gamma(X\times
Y\times\bb G_m^{\times n})$ and $p_i:X\times Y\times\bb G_m^{\times
n} \to\bb G_m$ is the projection. One sets
   $$\rho_{X,Y,n}(\Phi)= ((pr_Y)_*(\Phi), (pr_Y)_*([t_1]), \ldots, (pr_Y)_*([t_n]))$$
on objects and $\rho_{X,Y,n}(\Phi)(\alpha)=(pr_Y)_*(\alpha)$ on
morphisms.

}\end{defs}

The following lemma is a straightforward consequence of Corollary
\ref{propertiesf*andg*}.

\begin{lem}
The bivariant additive category
   $$\cc A:(\text{Sm}/k)^{\op}\times \text{Sm}/k\to AddCats,\quad(X,Y)\mapsto\cc A(X,Y),$$
satisfies the following property:

\begin{itemize}
\item[$(\ff{Aut})$] for every $X\in \text{Sm}/k,Y\in \text{Sm}/k$ and $n>0$,
the functors $\rho_{X,Y,n}$
meet the following two conditions: \\
(a) for any $f:X'\to X$ in $\text{Sm}/k$ and $n >0$ one has $f^*_n
\circ \rho_{X,Y,n}=\rho_{X',Y,n} \circ f^*$, where $f^*: \cc A(X,Y
\times\bb G_m^{\times n}) \to \cc A(X',Y \times\bb G_m^{\times n})$
is defined in Notation \ref{f^*andg_*};\\
(b) using Notation~\ref{f^*andg_*}, for any $g: Y \to Y'$ in $\text{Sm}/k$ and $n>0$ one has \\
   $$g_{*,n} \circ \rho_{X,Y,n}=\rho_{X,Y',n} \circ (g \times id_n)_*,$$
where $id_n$ is the identity morphism of $\bb G_m^{\times n}$.
\end{itemize}
\end{lem}

The following proposition is true as well.

\begin{prop}\label{CatIsomorpism}
For every $X\in \text{Sm}/k,Y\in \text{Sm}/k$ and $n>0$ the additive
functor
$$\rho_{X,Y,n}:\cc A(X,Y\times\bb G_m^{\times n})\to\cc A(X,Y)(\bb G_m^{\times n})$$
is a category isomorphism (it is not just an equivalence of categories).
\end{prop}

\section{Comparing $\cc A(X,Y)$ with $\tilde{\cc P}(X,Y)$}

Let $X,Y$ be two $k$-schemes of finite type over the base field $k$.
We denote by $\cc P(X,Y)$ the category of coherent $\cc O_{X\times
Y}$-modules $P_{X,Y}$ such that $\supp (P_{X,Y})$ is finite over $X$
and the coherent $\cc O_X$-module $(p_X)_*(P_{X,Y})$ is locally
free. A disadvantage of the category $\cc P(X,Y)$ is that whenever
we have two maps $f:X\to X'$ and $g:X'\to X''$ then the functor
$(g\circ f)^*$ agrees with $f^*\circ g^*$ only up to a canonical
isomorphism. To fix the problem, we replace $\cc P(X,Y)$ by the
equivalent additive category of big bimodules $\tilde{\cc P}(X,Y)$
which is functorial in both arguments. This is done in Appendix.

In this section for any $X\in \text{Sm}/k$ and $Y\in \text{AffSm}/k$
a canonical functor
   $$F_{X,Y}:\cc A(X,Y)\to\tilde{\cc P}(X,Y)$$
is constructed. Logically, one should now read Appendix about big
bimodules, and then return to this section.

As an application, we obtain canonical isomorphisms over a perfect
field $k$
   $$K_i(X)\cong DK_-^{\text{eff}}(k)(M_{\bb K}(X)[i],M_{\bb K}(pt)),\quad X\in \text{Sm}/k,i\in\bb Z,pt=\spec k,$$
where $K(X)$ is an algebraic $K$-theory spectrum defined as the
Waldhausen symmetric $K$-theory spectrum $K(\tilde{\cc P}(X,pt))$
and $DK_-^{\text{eff}}(k)$ is the triangulated category of
$K$-motives (see Definition~\ref{prekrasno}).

Let $X,Y \in \text{Sm}/k$ and assume that $Y$ is affine. Let $\cc
A(X,Y)$ be the additive category in the sense of
Definition~\ref{ccAobjects} and let $\tilde{\cc P}(X,Y)$ be the
additive category of big bimodules defined in Appendix. If
$f:X^{\prime}\to X$ is a morphism in $\text{Sm}/k$, then there is an
additive functor $f^*:\cc A(X,Y)\to\cc A(X^{\prime},Y)$ defined in
Notation~\ref{f^*andg_*}. By Corollary \ref{propertiesf*andg*} the
assignments $X \mapsto \cc A(X,Y)$ and $f\mapsto f^*$ yield a
presheaf of small additive categories on $\text{Sm}/k$. By
Lemma~\ref{derev} the assignments $X\mapsto\tilde{\cc P}(X,Y)$ and
$f\mapsto (f^*:\tilde{\cc P}(X,Y)\to\tilde{\cc P}(X^{\prime},Y))$
yield another presheaf of small additive categories on
$\text{Sm}/k$.

The main goal of this section is to prove the following

\begin{thm}\label{AandP}
Let $Y$ be an affine $k$-smooth variety. Then there is a morphism
  $$F:\cc A(-, Y)\to \tilde{\cc P}(-,Y)$$
of presheaves of additive categories on $\text{Sm}/k$ such that for
any $k$-smooth affine $X$ the functor $F_{X,Y}:\cc
A(X,Y)\to\tilde{\cc P}(X,Y)$ is an equivalence of categories.
\end{thm}

We postpone the proof but first construct a functor
   $$F_{X,Y}:\cc A(X,Y)\to\tilde{\cc P}(X,Y)$$
which is an equivalence of categories whenever $X$ is affine. We
shall do this in several steps. Let $\Phi\in\cc A(X,Y)$ be an
object. It is represented by a triple
   $$(n,Z,\phi:p_{X,*}(\cc O_Z)\to M_n(\cc O_X)),$$
where $n$ is a nonnegative integer, $Z\in\text{SubSch}(X\times Y/Y)$
(see Notation~\ref{SubSch}) and $\phi$ is a non-unital homomorphism
of sheaves of $\cc O_X$-algebras. Thus one can consider the
composite of non-unital $k$-algebra homomorphisms
   $$\Phi_X: k[Y] \to k[X \times Y] \to k[Z] \xrightarrow{\phi} M_n(k[X]).$$
Clearly, it does not depend on the choice of a triple representing
the object $\Phi$.

Let $Sch/X$ be the category of $X$-schemes of finite type. For an
$X$-scheme $f: U \to X$ in $Sch/X$ set
   $$\Phi_U:=M_n(f^*) \circ \Phi_X: k[Y] \to M_n(k[U]).$$
Note that $\Phi_U$ depends not only on $U$ itself but rather on the
$X$-scheme $U$. The assignment $U/X \mapsto \Phi_U$ defines a
morphism of presheaves of non-unital $k$-algebras $(U/X \mapsto
k[Y]) \to (U/X \mapsto M_n(k[U]))$.

One has a compatible family of projectors given by $U/X \mapsto
p^{\Phi}_U=\Phi_U(1) \in M_n(k[U])$. Set $P^{\Phi}_U=\im(p^{\Phi}_U)
\subset k[U]^n$. Then the assignment
   \begin{equation}\label{PPhi_U}
    U \mapsto P^{\Phi}_U
   \end{equation}
is a presheaf of $(U/X \mapsto k[U])$-modules.
Given $U/X \in Sch/X$ and a point $u \in U$, we set
   $$p^{\Phi}_{U,u}:=\lp_{u \in V \subset U}p^{\Phi}_V \in
   M_n(\mathcal O_{U,u}),\quad P^{\Phi}_{U,u}:=\im(p^{\Phi}_{U,u}) \subset\mathcal O_{U,u}^n.$$
The stalk of the presheaf $(U \mapsto P^{\Phi}_{U})$ of
$(U/X \mapsto k[U])$-modules
at the point $u \in U$ is the $\mathcal O_{U,u}$-module
$P^{\Phi}_{U,u}$.

The presheaf of $(U/X \mapsto k[U])$-modules $U/X \mapsto
P^{\Phi}_U$ has, moreover, a $k[Y]$-module structure. Namely, for
each $U/X \in Sch/X$ the $k$-algebra $k[Y]$ acts on the
$k[U]$-module $P^{\Phi}_U$ by means of the non-unital $k$-algebra
homomorphism $\Phi_U: k[Y] \to M_n(k[U])$.
Therefore the $k$-algebra $k[Y]$ acts on the $\mathcal
O_{U,u}$-module $P^{\Phi}_{U,u}$ by means of a non-unital
$k$-algebra homomorphism
   $$\Phi_{U,u}: k[Y] \xrightarrow{\Phi_U} M_n(k[U]) \xrightarrow{M_n(can)} M_n(\mathcal O_{U,u}),$$
where $can$ is
the localization homomorphism $k[U] \to \mathcal O_{U,u}$. {\it In
what follows we will regard the $\mathcal O_{U,u}$-module
$P^{\Phi}_{U,u}$ as an $\mathcal O_{U,u} \otimes_k k[Y]$-module via
the non-unital $k$-algebra homomorphism $\Phi_{U,u}$}.

\begin{defs}\label{StalkOfPhiUq}{\rm
Let $U/X \in Sch/X$ and $q \in U \times Y$ be a point. Let
$u=pr_U(q) \in U$ be its image in $U$. Set
   $$\mathcal P^{\Phi}_{U,q}:=\bigg\{\frac{m}{g}\mid m\in P^{\Phi}_{U,u}, g \in \mathcal O_{U,u}\otimes_k k[Y]\textrm{ such that } g(q) \neq 0\bigg\}/ \sim,$$
where "$\sim$" is the standard equivalence relation for fractions.
Clearly, $\mathcal P^{\Phi}_{U,q}$ is an $\mathcal O_{U \times Y,
q}$-module.

}\end{defs}

Now define a Zariski sheaf $\mathcal P^{\Phi}_U$ of $\mathcal O_{U
\times Y}$-modules on $U\times Y$ as follows. Its sections on an
open set $W\subset U\times Y$ are a compatible family of elements
$\{n_q\in \mathcal P^{\Phi}_{U,q}\}_{q \in  W}$. More precisely, we
give the following

\begin{defs}\label{PhiAndBigBimodule}{\rm
Set $\mathcal P^{\Phi}_U(W)$ to consist of the tuples
$(n_q)\in\prod_{q \in W}\mathcal P^{\Phi}_{U,q}$ such that there is
an affine cover $U=\cup U_i$ and for any $i$ there is an affine
cover of the form  $(W\cap{U_i\times Y})=\cup(U_i\times
Y)_{g_{i_j}}$ with $g_{i_j}\in k[U_i \times Y]$ and there are
elements $n_{i_j}\in(P^{\Phi}_{U_i})_{g_{i_j}}$ such that for any
$i$ and any $i_j$ and any point $q\in(U_i \times Y)_{g_{i_j}}$ one has
$n_{i_j}=n_q \in\mathcal P^{\Phi}_{U,q}$. Here $(U_i\times
Y)_{g_{i_j}}$ stands for the principal open set associated with
$g_{i_j}$.

Clearly, the assignment $W \mapsto \mathcal P^{\Phi}_U(W)$ is a
Zariski sheaf of $\mathcal O_{U \times Y}$-modules on $U\times Y$.
The $\mathcal O_{U \times Y}$-module structure on this sheaf is
given as follows: for $f \in k[W]$ and $(n_q)\in \mathcal
P^{\Phi}_U(W)$ set $f\cdot (n_q)=(f\cdot n_q)$.

}\end{defs}

Next, for any morphism $f:V\to U$ of objects in $Sch/X$ construct a
sheaf morphism
   $$\sigma_f:\mathcal P^{\Phi}_U\to F_*(\mathcal P^{\Phi}_V),$$
where $F=f \times \id: V \times Y \to U \times Y$. Given a point $v
\in V$ and its image $u \in U$, set $F^*_v=p^{\Phi}_{V,v} \circ f^*
\circ i^{\Phi}_{U,u}: P^{\Phi}_{U,u} \to P^{\Phi}_{V,v}$, where
$i^{\Phi}_{U,u}: P^{\Phi}_{U,u} \hookrightarrow \mathcal O_{U,u}^n$
is the inclusion.

For any point $r \in V \times Y$ and its image
$s=F(r) \in U \times Y$
set
$v=pr_V(r)$ and $u=pr_U(s)$.
Clearly, $f(v)=u$.
The $k$-algebra homomorphism
$\cc O_{U\times Y,s} \to \cc O_{V\times Y,r}$
makes
$\mathcal P^{\Phi}_{V,r}$
an $\cc O_{U\times Y,s}$-module.
There is a unique homomorphism
$F^*_{r}: \mathcal P^{\Phi}_{U,s} \to \mathcal P^{\Phi}_{V,r}$
of
$\mathcal O_{U \times Y,s}$-modules
making the diagram commutative
\begin{equation*}
    \xymatrix{P^{\Phi}_{U,u}\ar[r]\ar[d]_{F^*_{v}}&\mathcal P^{\Phi}_{U,s}\ar[d]^{F^*_{r}}\\
              P^{\Phi}_{V,v}\ar[r]_{}&\mathcal P^{\Phi}_{V,r}}
   \end{equation*}
Let $W \subset U \times Y$ be an open subset. By definition,
   $$\mathcal P^{\Phi}_U(W)=\{(n_s)\in\prod_{s \in W}\mathcal P^{\Phi}_{U,s}\mid n_s\textrm{ are locally compatible}\}$$
and
   $$F_*(\mathcal P^{\Phi}_V)(W)=\mathcal P^{\Phi}_V(F^{-1}(W))=
     \{(n_r)\in\prod_{r \in F^{-1}(W)}\mathcal P^{\Phi}_{V,r}\mid n_r \textrm{ are locally compatible}\}.$$
Define $F^*_W:\mathcal P^{\Phi}_U(W) \to F_*(\mathcal
P^{\Phi}_V)(W)$ as follows. Given a section $(n_s \in \mathcal
P^{\Phi}_{U,s})_{s\in W}$ of $\mathcal P^{\Phi}_U$ over $W$, set
   $$F^*_{W}((n_s \in \mathcal P^{\Phi}_{U,s})_{s \in W}):=((F^*_r(n_s)_{f(r)=s}))_{s \in W}.$$
It is straightforward to check that the assignment $W\mapsto
F^*_{W}$ defines an $\mathcal O_{U \times Y}$-sheaf morphism
$$\sigma_f: \mathcal P^{\Phi}_U \to F_*(\mathcal P^{\Phi}_V).$$
Moreover, for a pair of morphisms $g: U_3 \to U_2$ and $f: U_2 \to
U_1$ in $Sch/X$ one has
$$\sigma_{f\circ g}=(f\times\id_Y)_*(\sigma_g) \circ \sigma_f:
\mathcal P^{\Phi}_{U_1} \to (F \circ G)_*(\mathcal
P^{\Phi}_{U_3})=F_*(G_*(\mathcal P^{\Phi}_{U_3})).$$

\begin{lem}\label{objekt}
The data $U/X \mapsto \mathcal P^{\Phi}_{U}$ and $(f: V \to U)
\mapsto (\sigma_f: \mathcal P^{\Phi}_U \to F_*(\mathcal
P^{\Phi}_V))$ defined above determine an object of the category
$\tilde {\mathcal P}(X,Y)$. We shall denote this object by
$F_{X,Y}(\Phi)$.
\end{lem}

Now define the functor $F_{X,Y}:\cc A(X,Y)\to\tilde{\cc P}(X,Y)$ on
morphisms. Let $\alpha: \Phi \to \Psi$ be a morphism in $\cc
A(X,Y)$. The morphism $\alpha$ is a Zariski sheaf morphism
$$(U/X \mapsto P^{\Phi}_U) \to (U/X \mapsto P^{\Psi}_U)$$
on small Zariski site $X_{Zar}$ respecting the $k[Y]$-module
structure on both sides. We write $\alpha_U: P^{\Phi}_U \to
P^{\Psi}_U$ for the value of $\alpha$ at $U$. For any point $x \in
X$ the Zariski sheaf morphism $\alpha$ induces a morphism of stalks
   $$\alpha_x: P^{\Phi}_{X,x} \to P^{\Psi}_{X,x}.$$
Finally, for any point $q \in X \times Y$ and its image $x=p_X(q)
\in X$ one has a homomorphism
   $$\alpha_q: \mathcal P^{\Phi}_{X,q} \to \mathcal P^{\Psi}_{X,q}$$
given by $\alpha_q\bigl(\frac{m}{g}\bigr)= \frac{\alpha_x(m)}{g}$
for any $m \in P^{\Phi}_{X,x}$ and any $g \in \mathcal O_{X,x}
\otimes_k k[Y]$ with $g(q) \neq 0$.

\begin{defs}{\rm
Define a morphism $F_{X,Y}(\alpha): F_{X,Y}(\Phi) \to F_{X,Y}(\Psi)$
as follows. Given a Zariski open subset $W\subset X\times Y$ and a
section $s=(n_q)\in\cc P^{\Phi}_X(W)$, set
$\alpha_{W}(s)=(\alpha_q(n_q))$. Clearly, the family
$(\alpha_q(n_q))$ is an element of $\cc P^{\Psi}_X(W)$. Moreover,
$\alpha_{W}$ is a homomorphism and the assignment $W \mapsto
\alpha_{W}$ is a morphism in $\tilde {\cc P}(X,Y)$. We shall write
$F_{X,Y}(\alpha)$ for this morphism in $\tilde {\cc P}(X,Y)$.

}\end{defs}

\begin{lem}\label{AtoPtransform}
The assignments $\Phi \mapsto F_{X,Y}(\Phi)$ and $\alpha \mapsto
F_{X,Y}(\alpha)$ determine an additive functor $F_{X,Y}: \cc A(X,Y)
\to \tilde {\cc P}(X,Y)$. Moreover, for a given affine $k$-smooth
variety $Y$ the assignment $X \mapsto F_{X,Y}$ determines a morphism
of presheaves of additive categories.
\end{lem}

Lemma~\ref{AtoPtransform} implies that in order to prove
Theorem~\ref{AandP}, it remains to check that for affine $X,Y\in
\text{AffSm}/k$ the functor $F_{X,Y}$ is an equivalence of
categories. Firstly describe a plan of the proof. Given $X,Y\in
\text{AffSm}/k$ we shall construct a square of additive categories
and additive functors
\begin{equation}
    \xymatrix{\cc A(X,Y) \ar[r]^{F_{X,Y}}\ar[d]_{\Gamma} &\tilde{\cc P}(X,Y) \ar[d]^{R}\\
              A(X,Y)\ar[r]_{a_{X,Y}} &\cc P(X,Y)}
   \end{equation}
which commutes up to an isomorphism of additive functors. We shall
prove that the functors $\Gamma$, $a_{X,Y}$ and $R$ are equivalences
of categories. As a consequence, the functor $F_{X,Y}$ will be an
equivalence of categories.

\begin{defs}\label{catA}{\rm
For affine schemes $X,Y \in \text{AffSm}/k$ define a category
$A(X,Y)$ as follows. Objects of $A(X,Y)$ are the pairs $(n,\phi)$,
where $n \geq 0$ and $\phi: k[Y] \to M_n(k[X])$ is a non-unital
$k$-algebra homomorphism. The homomorphism $\phi$ defines a
projector $\phi(1) \in M_n(k[X])$. The projector $\phi(1)$ defines a
projective $k[X]$-module $\im\bigl(\phi(1): k[X]^n \to
k[X]^n\bigr)$. This $k[X]$-module has also a $k[Y]$-module structure
which is given by the non-unital homomorphism $\phi$. Namely, $mf:=
\phi(f)(m)$. Thus $\im(\phi(1))$ is a $k[X \times Y]$-module. Set
   $$\Mor_{A(X,Y)}((n_1, \phi_1),(n_1, \phi_1))= \Hom_{k[X \times Y]}(\im(\phi_1(1)), \im(\phi_2(1))).$$

}\end{defs}

\begin{defs}\label{ccAtoA}{\rm
Given affine schemes $X,Y\in \text{AffSm}/k$, define a functor
   $$\Gamma:\cc A(X,Y)\to A(X,Y)$$
as follows. Given an object $\Psi \in \cc A(X,Y)$, choose its
representative $(n,Z,\psi: p_{X,*}(\mathcal O_Z) \to M_n(\mathcal
O_X))$. This representative gives rise to a pair
   $$\Gamma(\Psi):=(n, \phi: k[Y] \to k[X \times Y] \to \Gamma(Z, \mathcal O_Z) \xrightarrow{\psi} M_n(k[X]) ),$$
which is an object of $A(X,Y)$. Clearly, this pair does not depend
on the choice of a representative.
One has an equality $\Gamma(X,P(\psi))= P^{\Psi}_X$, where $P(\psi)$
is defined in Definition~\ref{ccAobjects} and $P^{\Psi}_X$ is given
by~\eqref{PPhi_U}.
If $\alpha:\Psi_1\to\Psi_2$ is a morphism in $\cc A(X,Y)$, then
equalizing the supports of $\Psi_1$ and $\Psi_2$ and taking the
global sections on $X$, we get an isomorphism
   \begin{gather*}
    \Mor_{\cc A(X,Y)}(\Psi_1, \Psi_2)=\Hom_{p_{X,*}(\cc O_Z)}(P(\psi_1),P(\psi_2))\xrightarrow{\Gamma(\alpha)}\\
    \xrightarrow{\Gamma(\alpha)}\Hom_{k[X \times Y]}(P^{\Psi_1}_X,P^{\Psi_1}_X)=\Hom_{A(X,Y)}(\Gamma(\Psi_1), \Gamma(\Psi_2)).
   \end{gather*}
This completes the definition of the functor $\Gamma$.

}\end{defs}

\begin{lem}\label{ccAandA}
The functor $\Gamma: \cc A(X,Y) \to A(X,Y)$ is an equivalence of
additive categories.
\end{lem}

\begin{proof}
Define a functor $a: A(X,Y) \to \cc A(X,Y)$ on objects as follows.
An object $(n, \phi)$ in $A(X,Y)$ defines a projector $\phi(1) \in
M_n(k[X])$. The image $\im(\phi(1))$ in $k[X]^n$ has a $k[Y]$-module
structure given by the non-unital homomorphism $\phi$. In this way
$\im(\phi(1))$ is a $k[X \times Y]$-module. Let $A \subset X \times
Y$ be the support of $\im(\phi(1))$. Using Notation~\ref{Supp}, it
is easy to see that $A\in\supp(X \times Y/Y)$.
Thus there exists an
integer $m\geq 0$ such that $I^m_A \cdot\im(\phi(1))=(0)$. The
latter means that $\im(\phi(1))$ is a $k[X\times Y]/I^m_A$-module,
and therefore the non-unital $k$-algebra homomorphism $\phi$ can be
presented in the form
   $$k[X \times Y] \xrightarrow{can_{A,m}} k[X \times Y]/I^m_A \xrightarrow{\bar \phi_{A,m}} M_n(k[X])$$
for a unique $\bar\phi_{A,m}$. Let $Z=\spec(k[X \times Y]/I^m_A)$
and let $(\bar \phi_{A,m})^{\sim}: p_{X,*}(\mathcal O_Z) \to
M_n(\mathcal O_X)$ be the sheaf homomorphism associated to $\bar
\phi_{A,m}$. Set
   $$a(n,\phi):=\textrm{ the equivalence class of the triple }(n, Z, (\bar \phi_{A,m})^{\sim}).$$
This equivalence class remains unchanged when enlarging $A$ in
$\supp(X\times Y/Y)$ and the integer $m$. In fact, if $A^{\prime}
\in \supp(X \times Y/Y)$ is such that $A \subset A^{\prime}$ and
$m^{\prime} \geq m$, then $I^{m^{\prime}}_{A^{\prime}} \subset
I^m_A$. Thus $\phi=\bar \phi_{A^{\prime},m^{\prime}} \circ
can_{A^{\prime},m^{\prime}}$ for a unique
$\bar\phi_{A^{\prime},m^{\prime}}$. Let $Z^{\prime}=\spec(k[X \times
Y]/I^{m^{\prime}}_{A^{\prime}})$ and let $(\bar
\phi_{A^{\prime},m^{\prime}})^{\sim}: p_{X,*}(\mathcal
O_{Z^{\prime}}) \to M_n(\mathcal O_X)$ be the sheaf homomorphism
associated to $\bar \phi_{A^{\prime},m^{\prime}}$. Clearly, the
equivalence class of the triple $(n, Z, (\bar \phi_{A,m})^{\sim})$
coincides with the equivalence class of the triple $(n, Z^{\prime},
(\bar \phi_{A^{\prime},m^{\prime}})^{\sim})$.

Define the functor $a: A(X,Y) \to \cc A(X,Y)$ on morphisms as
follows. Let $\alpha: (n_1,\phi_1)\to (n_2,\phi_2)$ be a morphism in
$A(X,Y)$. Let $A_i$ be the support of the $k[X\times Y]$-module
$\im(\phi_i(1))$ and let $m_i$ be an integer such that $I^{m_i}_A
\cdot\im(\phi_i(1))=(0)$. Enlarging $A_1$ and $A_2$ in
$\supp(X\times Y/Y)$ if necessary, we may assume that $A_1=A=A_2$.
Enlarging $m_1$ and $m_2$, we may as well assume that $m_1=m=m_2$.
Therefore we may assume that $Z_1=Z=Z_2$. Now applying the functor
from $k[X]$-modules to $\mathcal O_X$-modules, we get a homomorphism
   \begin{gather*}
    \Hom_{A(X,Y)}((n_1, \phi_1), (n_1, \phi_1))=\Hom_{k[X \times Y]}(\im(\phi_1(1)),\im(\phi_2(1)))=\\
    =\Hom_{k[X\times Y]/I^m_A}(\im(\bar\phi_{1,A,m}(1)),\im(\bar\phi_{2,A,m}))\to\Hom_{p_{X,*}(\cc O_Z)}(\im(\bar
    \phi_1(1))^{\sim},\im(\bar \phi_2(1))^{\sim}).
   \end{gather*}
Set $a(\alpha)$ to be the image of $\alpha$ under this homomorphism.
The definition of the functor $a$ is completed.

It is straightforward to check that the functors $\Gamma$ and $a$
are mutually inverse equivalences of additive categories. For
instance, the composite $a \circ \Gamma$ is the identity functor
from $A(X,Y)$ to itself.
\end{proof}

\begin{defs}\label{a_XY}{\rm
Define a functor $a_{X,Y}: A(X,Y)\to\cc P(X,Y)$ as follows. It takes
an object $(n,\phi)$ to the $\mathcal O_{X\times Y}$-module sheaf
$\im(\phi(1))^{\sim}$ associated with the $k[X\times Y]$-module
$\im(\phi(1))$ described in Definition~\ref{catA}. On morphisms it
is defined by the isomorphism
   \begin{gather*}
    \Hom_{A(X,Y)}((n_1,\phi_1),(n_2,\phi_2))=\Hom_{k[X\times Y]}(\im(\phi_1(1)),\im(\phi_2(1)))\cong \\
    \cong \Hom_{\mathcal O_{X\times Y}}(\im(\phi_1(1))^\sim,\im(\phi_2(1))^\sim)=\Hom_{\cc P(X,Y)}(a_{X,Y}(n_1,\phi_1),a_{X,Y}(n_2,\phi_2)).
   \end{gather*}
}\end{defs}

\begin{proof}[Proof of Theorem~\ref{AandP}]
Consider the following square of functors
\begin{equation*}
    \xymatrix{\cc A(X,Y) \ar[r]^{F_{X,Y}}\ar[d]_{\Gamma} &\tilde{\cc P}(X,Y) \ar[d]^{R}\\
              A(X,Y)\ar[r]_{a_{X,Y}} &\cc P(X,Y)}
   \end{equation*}
where $R$ takes a big bimodule $P \in \tilde P(X,Y)$ to the
$\mathcal O_{X\times Y}$-module $P_{X,Y} \in\cc P(X,Y)$ and a
morphism $\alpha: P \to Q$ of big bimodules to the morphism
$\alpha_{X,Y}: P_{X,Y} \to Q_{X,Y}$ of $\mathcal O_{X\times
Y}$-modules. We claim that this diagram commutes up to an
isomorphism of functors. Since the functors $\Gamma$, $a_{X,Y}$, $R$
are equivalences of categories, the functor $F_{X,Y}$ is a category
equivalence, too. To complete the proof, it remains to construct a
functor isomorphism $a_{X,Y}\circ\Gamma\to R\circ F_{X,Y}$.

Let $\Psi \in \cc A(X,Y)$ be an object and let
$(n,Z,\psi:p_{X,*}(\cc O_Z)\to M_n(\cc O_X))$ be a triple
representing $\Psi$ (see Definition~\ref{ccAobjects}). Then
$\Gamma(\Psi)=(n, \phi: k[Y] \to k[X \times Y] \to \Gamma(Z,
\mathcal O_Z) \xrightarrow{\psi} M_n(k[X]))$ as described in
Definition~\ref{ccAtoA}. Let $\im(\phi(1))$ be the $k[X \times
Y]$-module described in Definition~\ref{catA}. Then
$a_{X,Y}(\Gamma(\Psi))$ is the $\mathcal O_{X\times Y}$-module sheaf
$\im(\phi(1))^{\sim}$ associated with the $k[X\times Y]$-module
$\im(\phi(1))$. On the other hand, following
Definition~\ref{PhiAndBigBimodule} and the description of $R$, one
has $R(F_{X,Y}(\Psi))=\mathcal P^{\Psi}_X$. We need to construct an
isomorphism $\theta_{\Psi}:
\im(\phi(1))^{\sim}\xrightarrow\cong\mathcal P^{\Psi}_X$, natural in
$\Psi$, of $\mathcal O_{X\times Y}$-modules. Giving such a morphism
$\theta_{\Psi}$ is the same as giving a $k[X \times Y]$-homomorphism
   $$\Theta_{\Psi}:\im(\phi(1))\to\Gamma(X\times Y,\mathcal P^{\Psi}_X).$$
Moreover, $\theta_{\Psi}$ is an isomorphism whenever so is
$\Theta_{\Psi}$. A section of $\mathcal P^{\Psi}_X$ on $X\times Y$
is a compatible family of elements $(n_q \in \mathcal
P^{\Phi}_{X,q})_{q \in X\times Y}$ (see
Definitions~\ref{PhiAndBigBimodule} and~\ref{StalkOfPhiUq}). For $s
\in\im(\phi(1))$, set
   $$\Theta_{\Psi}(s)=\biggl(\frac{s_{x(q)}}{1}\biggr)\in \prod_{{q \in X\times Y}}\mathcal P^{\Psi}_{X,q},$$
where $x(q)=p_X(q) \in X$ and $s_{x(q)} \in P^{\Psi}_{X,x(q)}$ is
the image of $s$ in $P^{\Psi}_{X,x(q)}$ under the canonical map
$P^{\Psi}_X=\im(p^{\Psi}_X)\to\im(p^{\Psi}_{X,x(q)})=
P^{\Psi}_{X,x(q)}$ (see the discussion above
Definition~\ref{StalkOfPhiUq}). Clearly, $\Theta_{\Psi}(s)$ belongs
to $\Gamma(X\times Y,\mathcal P^{\Psi}_X)$. We claim that
$\Theta_{\Psi}$ is an isomorphism. In fact, if
$\frac{s_{x(q)}}{1}=0$ for all $q \in X\times Y$ then $s=0$. It
follows that $\Theta_{\Psi}$ is injective. If
$(n_q)\in\Gamma(X\times Y,\mathcal P^{\Psi}_X)$, then
$(n_q)\in\prod_{{q \in X\times Y}}\mathcal P^{\Psi}_{X,q}$ is a
compatible family of elements. It follows from Definition
\ref{PhiAndBigBimodule} that there is a global section $s$ of the
sheaf $\im(\phi(1))^{\sim}$ such that for each $q \in X\times Y$ one
has $s_q=n_q$. Since $\Gamma(X\times Y,
\im(\phi(1))^{\sim})=\im(\phi(1))$ the map $\Theta_{\Psi}$ is
surjective. The fact that the assignment $\Psi \mapsto
\Theta_{\Psi}$ is a functor transformation $a_{X,Y} \circ \Gamma \to
R\circ F_{X,Y}$ is obvious. Our theorem now follows.
\end{proof}

Let $DK_-^{\text{\rm eff}}(k)$ be the triangulated category of
$K$-motives in the sense of Definition~\ref{prekrasno}. Recall that
the $\bb K$-motive $M_{\bb K}(X)$ of a $k$-smooth scheme $X$ is the
$\bb K$-module $C_*(\bb K(-,X))$ (see Definition~\ref{omotive} and
Notation~\ref{kux}). The $\bb K$-motive $M_{\bb K}(X)$ belongs to
the category $DK_-^{\text{\rm eff}}(k)$ as observed just below the
proof of Corollary \ref{porto}. To conclude the section, we give the
following application of Theorem~\ref{AandP}.

\begin{thm}\label{vesmaneploho}
Let $k$ be a perfect field and let $X$ be any scheme in
$\text{Sm}/k$. Then for every integer $i\in\bb Z$ there is a natural
isomorphism of abelian groups
   $$K_i(X)\cong DK_-^{\text{\rm eff}}(k)(M_{\bb K}(X)[i],M_{\bb K}(pt)),$$
where $K(X)$ is Quillen's $K$-theory of $X$.
\end{thm}

A priori, there is no reason for the right hand side to be zero for
$i<0$. However, Theorem~\ref{AandP} and the fact that $K$-theory of
$X$ is connective imply this is the case.

\begin{proof}
By Theorem~\ref{ooops} $\bb K$ is a strict $V$-spectral category.
By~\eqref{adjoint} one has a canonical isomorphism for every integer
$i$
   $$\shnis(k)(X[i],{\bb K}(-,pt))\cong\shnis\bb K(\bb K(-,X)[i],\bb K(-,pt)).$$
Let
   $$K:\text{Sm}/k\to Sp^\Sigma,\quad X\mapsto K(X)=K(\tilde{\cc P}(X,pt))$$
be the algebraic $K$-theory presheaf of symmetric spectra. It
follows from Theorem~\ref{AandP} that the natural map in
$Pre^\Sigma(\text{Sm}/k)$
   $$F:\bb K(-,pt)\to K,$$ induced by the additive functors
$F_{X,pt}:\cc A(X,pt)\to\tilde{\cc P}(X,pt)$, $X\in \text{Sm}/k$, is
a Nisnevich local weak equivalence.

Using Thomason's theorem~\cite{T} stating that algebraic $K$-theory
satisfies Nisnevich descent, we obtain isomorphisms
   $$K_i(X)\cong\shnis(k)(X[i],K)\cong\shnis\bb K(\bb K(-,X)[i],\bb K(-,pt)),\quad i\in\bb Z.$$

Consider a commutative diagram in $Pre^\Sigma(\text{Sm}/k)$
   \begin{equation*}\label{hop}
   \xymatrix{
    K(\cc A(-,pt))\ar[r]\ar[d]_F&K(\cc A(-,pt))_f\ar[r]\ar[d]_{\delta}&|\underline{\Hom}(\Delta^.,K(\cc A(-,pt))_f)|\ar[d]^\gamma\\
    K\ar[r]^\alpha&K_f\ar[r]^\beta&|\underline{\Hom}(\Delta^.,K_f)|}
   \end{equation*}
Here the lower $f$-symbol refers to a fibrant replacement functor in
the Nisnevich local model structure on $Pre^\Sigma(\text{Sm}/k)$.
Theorem~\ref{AandP} implies $F$ is a Nisnevich local weak
equivalence. By~\cite{T} $K(-)$ is Nisnevich excisive, and hence
$\alpha$ is a stable weak equivalence. Since $K(-)$ is homotopy
invariant, then $\beta$ is a stable weak equivalence. It follows
that $\delta,\gamma$ are stable weak equivalences. Therefore the
composition of the upper horizontal maps is a Nisnevich local weak
equivalence. Thus the canonical map
   $$\bb K(-,pt)\to M_{\bb K}(pt)$$
is a Nisnevich local weak equivalence. One has an isomorphism
   $$K_i(X)\cong\shnis\bb K(\bb K(-,X)[i],M_{\bb K}(pt)),\quad i\in\bb Z.$$
Since $\bb K(-,X)[i],M_{\bb K}(pt)$ are bounded below $\bb
K$-modules, then our theorem follows from Theorem~\ref{neploho}(2).
\end{proof}

\appendix\section{The category of big bimodules $\tilde{\cc P}(X,Y)$}

Let $X,Y$ be two schemes of finite type over the base field $k$. We
denote by $\cc P(X,Y)$ the category of coherent $\cc O_{X\times
Y}$-modules $P_{X,Y}$ such that $\supp (P_{X,Y})$ is finite over $X$
and the coherent $\cc O_X$-module $(p_X)_*(P_{X,Y})$ is locally
free. A disadvantage of the category $\cc P(X,Y)$ is that whenever
we have two maps $f:X\to X'$ and $g:X'\to X''$ then the functor
$(g\circ f)^*$ agrees with $f^*\circ g^*$ only up to a canonical
isomorphism. We want to replace $\cc P(X,Y)$ by an equivalent
additive category $\tilde{\cc P}(X,Y)$ which is functorial in both
arguments.

To this end, we use the construction which is in spirit like that of
Grayson for finitely generated projective modules~\cite{Gr} and
Friedlander--Suslin for big vector bundles~\cite{FS}. Let $X$ be a
Noetherian scheme. Consider the big Zariski site $Sch/X$ of all
schemes of finite type over $X$. We define the {\it category of big
bimodules\/} $\tilde{\cc P}(X,Y)$ as follows.

An object of $\tilde{\cc P}(X,Y)$ consists of the following data:

\begin{enumerate}
\item For any $U\in Sch/X$ one has a bimodule $P_{U,Y}\in\cc
P(U,Y)$.

\item For any morphism $f:U'\to U$ in $Sch/X$ one has a morphism
$\sigma_f:P_{U,Y}\to(f\times 1_Y)_*(P_{U',Y})$ in $\cc P(U,Y)$
satisfying:

\begin{itemize}

\item[(a)] $\sigma_1=1$.

\item[(b)] The morphism $\tau_f:(f\times 1_Y)^*(P_{U,Y})\to
P_{U',Y}$ which is adjoint to $\sigma_f$ must be an isomorphism in
$\cc P(U',Y)$.

\item[(c)] Given a chain of maps $U''\lra{f_1}U'\lra{f}U$ in $Sch/X$,
the following relation is satisfied
   $$\sigma_{f\circ f_1}=(f\times 1_Y)_*(\sigma_{f_1})\circ\sigma_f.$$
\end{itemize}
\end{enumerate}
A morphism of two big bimodules $\alpha:P\to Q$ is a morphism
$\alpha_{X,Y}:P_{X,Y}\to Q_{X,Y}$ in $\cc P(X,Y)$. Clearly, $\cc
P'(X,Y)$ is an additive category.

Given a map $g:X'\to X$ of two Noetherian schemes, we define an
additive functor
   $$g^*:\tilde{\cc P}(X,Y)\to\tilde{\cc P}(X',Y)$$
as follows. For any $U\in Sch/X'$ and $P\in\cc P(X,Y)$ one sets
$g^*(P)_{U,Y}=P_{U,Y}$, where $U$ is regarded as an object of
$Sch/X$ by means of composition with $g$. In a similar way, if
$h:U'\to U$ is a map in $Sch/X'$ then $\sigma_h:g^*(P)_{U,Y}\to
g^*(P)_{U',Y}$ equals $\sigma_h$. So we have defined $g^*$ on
objects. Let $\alpha:P\to Q$ be a morphism in $\tilde{\cc P}(X,Y)$.
By definition, it is a morphism $\alpha_{X,Y}:P_{X,Y}\to Q_{X,Y}$ in
$\cc P(X,Y)$. There is a commutative diagram
   $$\xymatrix{(g\times 1_Y)^*(P_{X,Y})\ar[d]_{(g\times 1_Y)^*(\alpha_{X,Y})}\ar[r]^(.65){\tau_g}&P_{X',Y}\ar[d]^{\alpha_{X',Y}}\\
               (g\times 1_Y)^*(Q_{X,Y})\ar[r]^(.65){\tau_g}&Q_{X',Y}}$$
where the horizontal maps are isomorphisms. Then
$g^*(\alpha):=\alpha_{X',Y}$. The functor $g^*$ is constructed.

\begin{lem}\label{derev}
Let $g_1:X''\to X'$ and $g:X'\to X$ be two maps of schemes. Then
$(g\circ g_1)^*=g_1^*\circ g^*$.
\end{lem}

\begin{proof}
This is straightforward.
\end{proof}

Now let us discuss functoriality in $Y$. For this consider a map
$h:Y\to Y'$. We construct an additive functor
   $$h_*:\tilde{\cc P}(X,Y)\to\tilde{\cc P}(X,Y')$$
in the following way. We set $h_*(P)_{U,Y'}=(1_U\times
h)_*(P_{U,Y})$ for any $P\in\tilde{\cc P}(X,Y)$. If $f:U'\to U$ is a
map in $Sch/X$ then
   $$(1_U\times h)_*(f\times 1_Y)_*=(f\times 1_{Y'})_*(1_{U'}\times h)_*.$$
We define $\sigma_f$ for $h_*(P)$ as
   $$(1_U\times h)_*(\sigma_f):(1_U\times h)_*(P_{U,Y})\to(1_U\times h)_*(f\times 1_Y)_*(P_{U',Y})
     =(f\times 1_{Y'})_*(1_{U'}\times h)_*(P_{U',Y}).$$
By definition, $h_*$ takes a morphism $\alpha_{X,Y}$ in $\cc P(X,Y)$
to $(1_X\times h)_*(\alpha_{X,Y})$. The construction of the functor
$h_*$ is completed.

\begin{lem}\label{deve}
Let $h_1:Y'\to Y''$ and $h:Y\to Y'$ be two maps of schemes. Then
$(h_1\circ h)_*=(h_1)_*\circ h_*$.
\end{lem}

\begin{proof}
This is straightforward.
\end{proof}

We leave the reader to verify the following
\begin{prop}\label{deve2}
The natural functor
   $$R: \tilde{\cc P}(X,Y)\to\cc P(X,Y),\quad P\mapsto P_{X,Y},$$
is an equivalence of additive categories.
\end{prop}
By
Lemmas~\ref{derev}-\ref{deve} $\tilde{\cc P}(X,Y)$ has the desired
functoriality properties in both arguments.


\begin{thebibliography}{99}

\bibitem{FS} E. M. Friedlander, A. Suslin, The spectral sequence relating algebraic K-theory to motivic
         cohomology, Ann. Sci. \'Ecole Norm. Sup. (4) 35 (2002), 773-875.
\bibitem{GP} G. Garkusha, I. Panin, K-motives of algebraic varieties, Homology Homotopy Appl. 14(2) (2012), 211-264.
\bibitem{GP1} G. Garkusha, I. Panin, On the motivic spectral sequence, preprint arXiv:1210.2242.
\bibitem{GH} T. Geisser, L. Hesselholt, Topological cyclic homology of
         schemes, in Algebraic K-theory (Seattle, WA, 1997), Proc. Sympos.
         Pure Math. 67, Amer. Math. Soc., Providence, RI, 1999, pp. 41-87.
\bibitem{Gr} D. Grayson, Weight filtrations via commuting automorphisms, K-Theory 9 (1995), 139-172.
\bibitem{HSS} M. Hovey, B. Shipley, J. Smith, Symmetric spectra, J. Amer. Math. Soc. 13 (2000), 149-208.
\bibitem{K96} B. Keller, Derived categories and their uses, In
         Handbook of Algebra, vol.~1, North-Holland, Amsterdam, 1996, pp.~671-701.
\bibitem{Kni} A. Knizel, Homotopy invariant presheaves with Kor-transfers, MSc Diploma, St. Petersburg State University, 2012.
\bibitem{Mor} F. Morel, The stable $\bb A^1$-connectivity theorems, K-theory 35 (2006), 1-68.
\bibitem{MV} F. Morel, V. Voevodsky, $\bb A^1$-homotopy theory of schemes,
         Publ. Math. IHES 90 (1999), 45-143.
\bibitem{Sch} S. Schwede, An untitled book project about symmetric spectra, available at www.math.uni-bonn.de/$\sim$schwede (version v3.0/April 2012).
\bibitem{SS} S. Schwede, B. Shipley, Stable model categories are categories of modules, Topology 42(1) (2003), 103-153.
\bibitem{SV1} A. Suslin, V. Voevodsky, Bloch--Kato conjecture and motivic cohomology with
         finite coefficients, The Arithmetic and Geometry of Algebraic Cycles
         (Banff, AB, 1998), NATO Sci. Ser. C Math. Phys. Sci., Vol. 548,
         Kluwer Acad. Publ., Dordrecht (2000), pp. 117-189.
\bibitem{T} R. W. Thomason, T. Trobaugh, Higher algebraic K-theory of schemes and of
         derived categories, The Grothendieck Festschrift~III, Progress in
         Mathematics 88, Birkh\"auser, 1990, pp.~247-435.
\bibitem{Voe} V. Voevodsky, Cohomological theory of presheaves with transfers,
         in Cycles, Transfers and Motivic Homology Theories (V. Voevodsky, A.
         Suslin and E. Friedlander, eds.), Annals of Math. Studies, Princeton University Press, 2000.
\bibitem{Voe1} V. Voevodsky, Triangulated category of motives over a
         field, in Cycles, Transfers and Motivic Homology Theories (V.
         Voevodsky, A. Suslin and E. Friedlander, eds.), Annals of Math.
         Studies, Princeton University Press, 2000.
\bibitem{Wal} F. Waldhausen, Algebraic K-theory of spaces, In Algebraic and
         geometric topology, Proc. Conf., New Brunswick/USA 1983, Lecture
         Notes in Mathematics, No.~1126, Springer-Verlag, 1985, pp.~318-419.
\bibitem{Wlk} M. E. Walker, Motivic cohomology and the K-theory of
         automorphisms, PhD Thesis, University of Illinois at Urbana-Champaign, 1996.
\end{thebibliography}
\end{document}